\documentclass[12pt]{amsart}
\usepackage{amsthm,amsmath,amssymb,amsbsy,amsfonts,latexsym,amsopn,amstext,cite,
                                               amsxtra,euscript,amscd,bm,mathabx,mathrsfs, abraces}
\usepackage{url}
\usepackage[colorlinks,linkcolor=blue,anchorcolor=blue,citecolor=blue,backref=page]{hyperref}
\usepackage{color}
\usepackage{graphics,epsfig}
\usepackage{graphicx}
\usepackage{float} 
\usepackage[english]{babel}
\usepackage{mathtools}
\usepackage{todonotes}
\usepackage{url}
\usepackage[colorlinks,linkcolor=blue,anchorcolor=blue,citecolor=blue,backref=page]{hyperref}

\usepackage[norefs,nocites]{refcheck}

\hypersetup{breaklinks=true}

\usepackage[norefs,nocites]{refcheck}
\usepackage[english]{babel}
\begin{document}

\newtheorem{thm}{Theorem}
\newtheorem{lem}[thm]{Lemma}
\newtheorem{claim}[thm]{Claim}
\newtheorem{cor}[thm]{Corollary}
\newtheorem{prop}[thm]{Proposition} 
\newtheorem{definition}[thm]{Definition}
\newtheorem{question}[thm]{Open Question}
\newtheorem{conj}[thm]{Conjecture}
\newtheorem{rem}[thm]{Remark}
\newtheorem{prob}{Problem}

\def\ccr#1{\textcolor{red}{#1}}
\def\cco#1{\textcolor{orange}{#1}}
\def\ccg#1{\textcolor{cyan}{#1}}

\newtheorem{ass}[thm]{Assumption}

\newtheorem{lemma}[thm]{Lemma}

\newcommand{\GL}{\operatorname{GL}}
\newcommand{\SL}{\operatorname{SL}}
\newcommand{\lcm}{\operatorname{lcm}}
\newcommand{\ord}{\operatorname{ord}}
\newcommand{\Tr}{\operatorname{Tr}}
\newcommand{\Span}{\operatorname{Span}}

\numberwithin{equation}{section}
\numberwithin{thm}{section}
\numberwithin{table}{section}

\def\vol {{\mathrm{vol\,}}}
\def\squareforqed{\hbox{\rlap{$\sqcap$}$\sqcup$}}
\def\qed{\ifmmode\squareforqed\else{\unskip\nobreak\hfil
\penalty50\hskip1em\null\nobreak\hfil\squareforqed
\parfillskip=0pt\finalhyphendemerits=0\endgraf}\fi}

\def \balpha{\bm{\alpha}}
\def \bbeta{\bm{\beta}}
\def \bgamma{\bm{\gamma}}
\def \blambda{\bm{\lambda}}
\def \bchi{\bm{\chi}}
\def \bphi{\bm{\varphi}}
\def \bpsi{\bm{\psi}}
\def \bomega{\bm{\omega}}
\def \btheta{\bm{\vartheta}}
\def \bmu{\bm{\mu}}
\def \bnu{\bm{\nu}}

\newcommand{\bfxi}{{\boldsymbol{\xi}}}
\newcommand{\bfrho}{{\boldsymbol{\rho}}}

\def\cA{{\mathcal A}}
\def\cB{{\mathcal B}}
\def\cC{{\mathcal C}}
\def\cD{{\mathcal D}}
\def\cE{{\mathcal E}}
\def\cF{{\mathcal F}}
\def\cG{{\mathcal G}}
\def\cH{{\mathcal H}}
\def\cI{{\mathcal I}}
\def\cJ{{\mathcal J}}
\def\cK{{\mathcal K}}
\def\cL{{\mathcal L}}
\def\cM{{\mathcal M}}
\def\cN{{\mathcal N}}
\def\cO{{\mathcal O}}
\def\cP{{\mathcal P}}
\def\cQ{{\mathcal Q}}
\def\cR{{\mathcal R}}
\def\cS{{\mathcal S}}
\def\cT{{\mathcal T}}
\def\cU{{\mathcal U}}
\def\cV{{\mathcal V}}
\def\cW{{\mathcal W}}
\def\cX{{\mathcal X}}
\def\cY{{\mathcal Y}}
\def\cZ{{\mathcal Z}}
\def\Ker{{\mathrm{Ker}}}

\def\sA{{\mathscr A}}

\def\NmQR{N(m;Q,R)}
\def\VmQR{\cV(m;Q,R)}

\def\Xm{\cX_m}

\def \A {{\mathbb A}}
\def \B {{\mathbb A}}
\def \C {{\mathbb C}}
\def \F {{\mathbb F}}
\def \G {{\mathbb G}}
\def \L {{\mathbb L}}
\def \K {{\mathbb K}}
\def \Q {{\mathbb Q}}
\def \Qbar {\overline{\mathbb Q}}
\def \R {{\mathbb R}}
\def \Z {{\mathbb Z}}
\def \P {{\mathbb P}}

\def \fA{\mathfrak A}
\def \fC{\mathfrak C}
\def \fL{\mathfrak L}
\def \fR{\mathfrak R}
\def \fS{\mathfrak S}

\def \fUg{{\mathfrak U}_{\mathrm{good}}}
\def \fUm{{\mathfrak U}_{\mathrm{med}}}
\def \fV{{\mathfrak V}}
\def \fG{\mathfrak G}
\def \f{\mathfrak G}

\def\e{{\mathbf{\,e}}}
\def\ep{{\mathbf{\,e}}_p}
\def\eq{{\mathbf{\,e}}_q}

 \def\\{\cr}
\def\({\left(}
\def\){\right)}
\def\fl#1{\left\lfloor#1\right\rfloor}
\def\rf#1{\left\lceil#1\right\rceil}

\def\Im{{\mathrm{Im}}}

\def \oF {\overline \F}

\newcommand{\pfrac}[2]{{\left(\frac{#1}{#2}\right)}}

\def \Prob{{\mathrm {}}}
\def\e{\mathbf{e}}
\def\ep{{\mathbf{\,e}}_p}
\def\epp{{\mathbf{\,e}}_{p^2}}
\def\em{{\mathbf{\,e}}_m}

\def\Res{\mathrm{Res}}
\def\Orb{\mathrm{Orb}}

\def\vec#1{\mathbf{#1}}
\def \va{\vec{a}}
\def \vb{\vec{b}}
\def \vc{\vec{c}}
\def \ve{\vec{e}}
\def \vs{\vec{s}}
\def \vu{\vec{u}}
\def \vv{\vec{v}}
\def \vw{\vec{w}}
\def \vell{\boldsymbol{\ell}}
\def\vlam{\vec{\lambda}}
\def\flp#1{{\left\langle#1\right\rangle}_p}

\def\Mat{\cM}

\def\mand{\qquad\mbox{and}\qquad}

\title[Multiplicative dependence of matrices]
{Integer matrices with a given characteristic polynomial and multiplicative dependence of matrices}

\author[P. Habegger] {Philipp Habegger}
\address{Department of Mathematics and Computer Science, University of Basel, 
Spiegelgasse 1, 4051 Basel, Switzerland}
\email{philipp.habegger@unibas.ch}

\author[A. Ostafe] {Alina Ostafe}
\address{School of Mathematics and Statistics, University of New South Wales, Sydney NSW 2052, Australia}
\email{alina.ostafe@unsw.edu.au}

\author[I. E. Shparlinski] {Igor E. Shparlinski}
\address{School of Mathematics and Statistics, University of New South Wales, Sydney NSW 2052, Australia}
\email{igor.shparlinski@unsw.edu.au}

\begin{abstract}  We consider the set  $\cM_n\(\Z; H\)$ of $n\times n$-matrices with 
integer elements of size at most $H$ and obtain  a new upper bound on the number of matrices  from $\cM_n\(\Z; H\)$
with a given characteristic polynomial $f \in \Z[X]$,  which is uniform with respect to $f$. This complements 
the asymptotic formula of A.~Eskin, S.~Mozes and  N.~Shah (1996) in which $f$ has to be fixed and irreducible. 

Using this result, among others, we obtain upper and lower bounds on the number of $s$-tuples
of matrices  
 from $\cM_n\(\Z; H\)$, satisfying  various multiplicative relations, including 
  multiplicative  dependence and 
 bounded generation of a subgroup of $\GL_n(\Q)$.  These problems generalise those studied 
in the scalar case $n=1$ by F.~Pappalardi, M.~Sha, I.~E.~Shparlinski and C.~L.~Stewart (2018) with an 
obvious distinction due to the non-commutativity of matrices.   

Motivated by these problems, we also prove various properties of the variety of complex matrices with fixed characteristic polynomial, including computing the degree of this variety. %% in Appendix~\ref{app: deg Vf}.
\end{abstract}



\maketitle

\tableofcontents

%---------------------------------------------------------------------
\section{Introduction}
\subsection{Set-up and motivation}
For a positive integer $n$  let $\cM_n\(\Z\)$  denote the
set of all  $n\times n$ matrices with integer elements. 
Furthermore, for  a real $H\ge 1$ we use $\cM_n\(\Z; H\)$  to denote the
set  of matrices
\[
A = (a_{ij})_{i,j=1}^n \in \cM_n\(\Z\)
\]  
with integer entries of size $|a_{ij}| \le H$.  In particular, $\cM_n\(\Z; H\)$ is 
of cardinality $\# \cM_n\(\Z; H\) = \(2\fl{H} +1\)^{n^2}$.

Here we consider some questions of {\it arithmetic statistics\/} with matrices from 
$\cM_n\(\Z; H\)$ when $n$ is fixed and $H\to \infty$. We note that the dual situation, 
where the entries of matrices are drawn from a small fixed set, for example, from $\{0,1\}$ 
but $n\to \infty$ has recently received 
a lot of attention, see~\cite{BVW, CJMS, Eber, FLM, FJSS, JSS, Kop, KwSa, NgPa, NgWo, TaVu, Wo} and references therein as well as the recent survey~\cite{Vu}. The case we investigate here, that is, of fixed $n$ and $H\to \infty$,  has also been 
investigated, for example, in~\cite{DRS,E-BLS,EskKat,EMS, Kat2,Kat-Sing} 
but it seems to be significantly less studied.  Thus we hope that this work may attract more interest to this direction. 

  \subsection{Matrices with a given characteristic polynomial} 
Motivated by various applications, including studying the multiplicative structure of matrices as described below, 
we obtain an upper bound on the number $R_{n}(H;f)$
  of matrices $A \in \cM_n\(\Z; H\)$ with a given characteristic polynomial $f \in \Z[X]$, 
 defined as $f(X) = \det(XI_n - A)$,  which is
  of course a question of independent interest. 
  
In the case when  
 $f \in \Z[X]$  is   a monic irreducible polynomial, Eskin,  Mozes and   Shah~\cite{EMS} have given an asymptotic formula for a variant 
 $\widetilde R_{n} (H;f)$
 of $R_{n} (H;f)$, where  the matrices are ordered 
by the $L_2$-norm rather than by the $L_\infty$-norm, but this should not be very essential and plays no
role in our context as we are only interested in upper bounds for $R_{n} (H;f)$. 
Namely, by~\cite[Theorem~1.3]{EMS},
\begin{equation}
\label{eq:Asymp}
\widetilde R_{n} (H;f) = (C(f) + o(1)) H^{n(n-1)/2},
\end{equation}
with some constant $C(f)>0$ depending on $f$, 
when a monic  irreducible polynomial  $f \in \Z[X]$  is fixed. 
This result of~\cite{EMS} as
well as its variants  obtained via different approaches in~\cite{Shah, WeXu} (see also the result of~\cite{ShapZhe} in the case when $f$ splits completely over $\Q$)  are not sufficient for our purposes because we need an upper bound which: 
\begin{itemize}
\item   holds for arbitrary  $f \in \Z[X]$, which is not necessary irreducible;
\item   is uniform with respect to the coefficients of $f$.
\end{itemize}

The results of~\cite{EMS,Shah, WeXu} lead us towards
the following conjecture. 

\begin{conj}
\label{conj:CharPoly}
Uniformly over polynomials $f$ of degree $n$ we have 
\[
R_{n}(H;f) \le H^{n(n-1)/2+o(1)},
\]
as $H \to \infty$.   
\end{conj}

In particular, the bound of Conjecture~\ref{conj:CharPoly} holds for  any fixed 
polynomial $f$ of the special type considered in~\cite{EMS,Shah, WeXu}.

Unconditionally, only counting matrices with a  given determinant and 
applying~\cite[Theorem~4]{Shp}, see also  Lemma~\ref{lem: Det} below,
we instantly obtain 
\begin{equation}
\label{eq:TrivBound}
R_{n} (H;f) \le H^{n^2 - n+o(1)}.
\end{equation}
Furthermore, we also show that  Conjecture~\ref{conj:CharPoly}  holds for $n =2$, see Theorem~\ref{thm:gamma23}.

We define $\gamma_n$ as the largest real number 
such that 
uniformly over polynomials $f$ we have 
\begin{equation}
\label{eq:CharPoly}
R_{n} (H;f) \le H^{n^2-n -\gamma_n +o(1)},
\end{equation}
as $H \to \infty$. 

We first note that 
\begin{equation}
\label{eq:gamma_n low}
\gamma_n\le n(n-1)/2.
\end{equation} 
Indeed, to see this we count the number of possible characteristic polynomials $X^n+\ell_1 X^{n-1}+\cdots+\ell_n$ of matrices in $\cM_n\(\Z;H\)$, which is $O\(H^{n(n+1)/2}\)$ since  $\ell_i\ll H^{i}$, $i=1,\ldots,n$. Therefore, for some characteristic polynomial $f$, 
\[
R_{n} (H;f) \gg H^{n^2-n(n+1)/2}=H^{n^2-n-n(n-1)/2},
\]
which implies the upper bound on $\gamma_n$ (the notations $\ll$ and $\gg$ are defined in Section~\ref{sec:not}).
Clearly  the value  $\gamma_n = n(n-1)/2$ corresponds to Conjecture~\ref{conj:CharPoly} 
while by~\eqref{eq:TrivBound} we have $\gamma_n \ge 0$.

We remark that even with this trivial inequality $\gamma_n\ge 0$  our results are still nontrivial, but 
of course improve when $\gamma_n$ grows.

In particular, we are able to prove that~\eqref{eq:CharPoly} holds with  some explicit strictly positive lower bound on $\gamma_n$, see Theorems~\ref{thm:gamma23} and~\ref{thm:gamman} below. 

Finally, we recall that the question of counting  matrices with a given characteristic polynomial
over a finite field has been completely solved by  Gerstenhaber~\cite{Ger} and Reiner~\cite{Rei}, see also Lemma~\ref{lem:Rei}. 

\subsection{Multiplicative relations between matrices} 
The  
specific pro\-blems 
we investigate here have a common underlying motif  of  various multiplicative relations between matrices from $\cM_n\(\Z; H\)$. 
Some of the results we obtain depend on the quality of the estimates in Theorems~\ref{thm:gamma23} and~\ref{thm:gamman} below.

We say that an $s$-tuple of {\it non-singular\/} matrices $(A_1, \ldots, A_s)$ is {\it multiplicatively 
dependent\/} if there is a non-zero vector $(k_1, \ldots, k_s)\in \Z^s$ such that 
\begin{equation}
\label{eq:MultDep}
A_1^{k_1}\ldots A_s^{k_s} = I_n,
\end{equation}
where $I_n$ is the $n\times n$ identity matrix.  

Our methods often apply to singular matrices as well when assuming $k_i\ge 0$ in~\eqref{eq:MultDep} 
 and with the standard convention $A^0= I_n$ for any $n\times n$ matrix $A$. 
However, to avoid unnecessary discussion of this case we limit 
ourselves to only non-singular matrices.

In particular, our motivation to study multiplicatively dependent matrices comes from 
recent work on multiplicatively dependent integers (and also algebraic integers), see~\cite{PSSS,KSSS,St}.
The matrix version of this problem is however of very different spirit and requires different 
methods. The most obvious distinction between the matrix and scalar cases is of course 
the non-commutativity of matrix multiplication.  In particular, the property of multiplicative dependence may change 
if the entries of  $(A_1, \ldots, A_s)$  are permuted. However,  since we always assume that $s$ is fixed, allowing permutations of $A_1, \ldots, A_s$ 
in our definition of multiplicative dependence would not 
affect our bounds  in Theorems~\ref{thm:MaxRank} and~\ref{thm:MultDep} of such tuples. 
Another important distinction is the lack of one of the main tools of~\cite{PSSS}, namely the existence and uniqueness of prime number factorisation.

We note that the notion of multiplicative dependence given by~\eqref{eq:MultDep} is also 
motivated by the notion of {\it bounded generation\/}. Namely, we say that a subgroup $\Gamma\subseteq \GL_n(\Q)$ is boundedly generated if for some 
$A_1, \ldots, A_s \in \GL_n(\Q)$ we have
\[
\Gamma = \{A_1^{k_1}\ldots A_s^{k_s}:~k_1, \ldots, k_s \in \Z\}, 
\]
see~\cite{CRRZ} and references therein.   
For a non-commutative group, 
 boundedly generated is a  stronger condition than {\it finitely generated\/}.

Here we use a different approach to establish nontrivial upper and lower bounds on 
the cardinality of the set  $\cN_{n,s}(H)$ of multiplicatively dependent 
$s$-tuples $(A_1, \ldots, A_s) \in \cM_n\(\Z; H\)^s$. 
It is  convenient and natural to consider the subset  $\cN_{n,s}^*(H)$
of $\cN_{n,s}(H)$ which consists of multiplicatively dependent 
$s$-tuples $(A_1, \ldots, A_s) \in \cN_{n,s}(H)$ of maximal rank  $s-1$, 
 that is, such that 
any sub-tuple $\(A_{i_1}, \ldots, A_{i_t}\)$ of length $t < s$ with $1 \le i_1 < \ldots  <i_t \le s$ is multiplicatively independent.

Furthermore, the non-commutativity of matrices suggests yet another variation of the 
above questions. Namely,  we say that an $s$-tuple $(A_1, \ldots, A_s) \in \cM_n\(\Z; H\)^s$ 
of non-singular matrices 
is {\it free\/} if 
\[
A_{i_1}^{\pm 1} \cdots A_{i_L}^{\pm 1} \ne I_n
\]
for any  nontrivial 
reduced word   in $A_1^{\pm 1}, \ldots, A_s^{\pm 1}$, that is, a word without  
occurrences of the form $A_i A_i^{-1}$,  of any length  $L \ge 1$. Unfortunately we do not have nontrivial upper bounds 
on  the number of non-free  $s$-tuples $(A_1, \ldots, A_s) \in \cM_n\(\Z; H\)^s$, unless when $n=2$ and $A_i\in \SL_2(\Z;H)$, in which case matching lower and upper bounds are given in~\cite[Corollary~1.2]{BOS}. 
However we can estimate the number of such $s$-tuples  with the additional
condition  that if 
\begin{equation}
\label{eq:Nonfree}
A_{i_1}^{\pm 1} \cdots A_{i_L}^{\pm 1} = I_n,\quad i_1,\ldots,i_L\in \{1,\ldots,s\},
\end{equation} 
then for at least one $i =1, \ldots, s$ the $\pm 1$ exponents of $A_i$ 
in~\eqref{eq:Nonfree}  do not sum up to zero.  In particular, we denote by $\cK_{n,s}(H)$
the set of $s$-tuples $(A_1, \ldots, A_s) \in \cM_n\(\Z; H\)^s$ which satisfy the above condition.

To understand the meaning of this condition
we recall that the abelianisation  of a group $\cG$ is the map 
\begin{equation}
\label{eq:AbMap}
\cG \to \cG/[\cG,\cG],
\end{equation}
where $[\cG,\cG]$ is the commutator subgroup of $\cG$. 
In particular, all words $g_{i_1}^{\pm 1} \cdots g_{i_L}^{\pm 1}$, where 
for every $i$, the exponents of $g_i$ sum up to zero, are mapped to the 
identity element of the factor group $ \cG/[\cG,\cG]$ and thus are in the
kernel of this map.   Hence, in other words, $\cK_{n,s}(H)$ is
the set of $s$-tuples $(A_1, \ldots, A_s) \in \cM_n\(\Z; H\)^s$ such that the kernel of the 
abelianisation map~\eqref{eq:AbMap} of the 
group they generate contains also words which are not of this
type.

\subsection{Notation and conventions}
\label{sec:not}
We recall that  the notations $U = O(V)$, $U \ll V$ and $ V\gg U$  
are equivalent to $|U|\leqslant c V$ for some positive constant $c$, 
which throughout this work, may depend only on $n$ and $s$, and 
also on the  the degree, dimension and the number of variables and 
polynomials defining varieties in Section~\ref{sec:rat point}.

We also write $U = V^{o(1)}$ if for any fixed $\varepsilon>0$  we have 
$V^{-\varepsilon} \le |U |\le V^{\varepsilon}$ provided that $V$ is large 
enough.  Similarly $U = (1+ o(1))V$ means that  for any fixed $\varepsilon>0$ we have 
$|U-V| \le  \varepsilon |V|$ provided that $V$ is large enough.

For a finite set $\cS$ we use $\# \cS$ to denote its cardinality. 

For an integer $k\ne 0$ we denote by $\tau(k)$ the number of  positive
integral divisors of $k$, 
for which we very often use the well-known bound
\begin{equation}
\label{eq:tau}
\tau(k) = |k|^{o(1)} 
\end{equation}
as $|k| \to \infty$, see~\cite[Equation~(1.81)]{IwKow}.

We always assume that the variables of summation run over integers. Thus, for example
\[
\sum_{u \in [-H,H]} *\!*\!* \quad  = \quad  \sum_{u \in [-H,H]\cap\Z} *\!*\!*,
\]
and similarly for multivariate sums. 

When we say that something holds uniformly over polynomials $f$, 
we mean that neither the implied constants nor the decay of the quantity $o(1)$
depend on a choice of $f$ (and similarly for other parameters). 

Furthermore, in all our asymptotic statements, for the involving $o(1)$ we always 
assume that  $H \to \infty$. 

As usual for a matrix (or a vector) $A$ we use $A^t$ to denote the transposition of $A$.

  \section{Main results}

\subsection{Counting  matrices with a given characteristic polynomial}

We start with bounding $\gamma_n$ which appears in~\eqref{eq:CharPoly} since 
this parameter appears in several other results of the paper and of course such 
bounds are of independent interest. 

We are able to show the following results towards  Conjecture~\ref{conj:CharPoly}. 
We start with the cases of $n =2,3$.

\begin{thm} \label{thm:gamma23} 
We have 
\[
\gamma_2 = 1 \mand \gamma_3 \ge 2.
\]
\end{thm}    

Finally, for $n \ge 4$ we have a slightly weaker result. 

\begin{thm} \label{thm:gamman} 
For $n \ge 4$, we have 
\[
\gamma_n \ge   1. 
\]
\end{thm} 

\begin{rem} 
\label{rem:dim growth}
It may be possible that the recent dimension growth
result~\cite[Theorem~1.2]{Verm} for affine varieties can provide 
an alternative proof of Theorem~\ref{thm:gamman}. One would need to 
  establish   that the variety of matrices with fixed characteristic
  polynomial is not cylindrical over a curve. Our method
  exploits the structure of the matrices in question. We
  hope that it can possibly lead to further improvements and generalisations. 
\end{rem}

\begin{rem}
We also notice that for $n\ge 2$, applying Proposition~\ref{prop:degreenfac} in Appendix~\ref{app: deg Vf} and~\cite[Theorem~A]{Pila}, one can easily obtain  the bound
\[
\gamma_n \ge   1-\frac{1}{n!}. 
\]
Furthermore,  thanks to the uniformity of~\cite[Theorem~A]{Pila}, it  implies that for any $n\times n$ matrix $K$ with 
real entries (of arbitrary size) and any polynomial $f \in \Z[X]$, the number of matrices
in the shifted set $K + \cM_n\(\Z; H\)$, for which $f$ is the characteristic polynomial, is bounded by $O\(H^{n^2 - n - 1 + 1/n!}\)$, 
uniformly over $K$ and $f$. We would like to emphasise that we do not impose any integrality conditions on $K$. However, 
we have to assume  $f \in \Z[X]$ since Propositions~\ref{prop:irred scheme} and~\ref{prop:degreenfac} are  
only established for integers polynomials (which stems from the reduction argument in the proof of Lemma~\ref{lem:IrredDimDeg}).
\end{rem}

\subsection{Counting multiplicatively dependent matrices} 
\label{sec:md matr}
We begin with   bounds on the size of $\cN_{n,s}^*(H)$.
 
\begin{thm}
\label{thm:MaxRank}
We have 
\[
  H^{ sn^2  - \rf{s/2} n+o(1)} \ge  \# \cN_{n,s}^*(H)  \ge
 \begin{cases} H^{(s-1)n^2/2 +n/2 + o(1)}, & \text{if $s$ is even},\\
 H^{(s-1)n^2/2  + o(1)}, & \text{if $s$ is odd}.
 \end{cases} 
\]
\end{thm}

Next, we use  Theorem~\ref{thm:MaxRank} to obtain an upper bound for  $\# \cN_{n,s}(H)$.

To present a lower bound we need to recall that an integer $m$ is called a {\it totient\/} if it is 
a value of the Euler function $m = \varphi(k)$ for some integer $k$. The
best known result about the density of the set of totients is due to Ford~\cite{Ford1,Ford2}. However for us it is more important to control the size of the gaps between totients. 
Namely, let $v(n)$ be the largest possible totient $m\le n$.  Note that by a result 
of  Baker,  Harman and  Pintz~\cite{BaHaPi} on gaps between primes we immediately conclude that
\begin{equation}
\label{eq:gap tot}
n \ge v(n) \ge n - n^{21/40},
\end{equation}
for all sufficiently large $n$. 
We use this as an opportunity to attract attention to the question of bounding the 
gaps between totients
in a better way, which is of independent interest, see also~\cite{FoKo}
for some related results.

Since $1=\varphi(1)$ is a totient, each integer can be represented as a sum 
of some number $h \ge 1$ of totients and hence we can define
\begin{equation}
\label{eq:wn}
w(n) = \max\left \{ \sum_{j=1}^h \varphi(k_j)^2:~n =\sum_{j=1}^h \varphi(k_j)\right\},
\end{equation}
where the maximum is taken over all such representations of all possible lengths $h \ge 1$. 
In particular, by~\eqref{eq:gap tot} we have a trivial bound 
\[
n^2 \ge w(n) \ge \(n - n^{21/40}\)^2
\]
for all sufficiently large $n$.   

We have the following estimates

\begin{thm}
\label{thm:MultDep} With $\gamma_n$ as in~\eqref{eq:CharPoly} and $w(n)$ as in~\eqref{eq:wn}, 
we have 
\[
 H^{ sn^2  - n -\min\{n,\gamma_n\}+o(1)}   \ge \# \cN_{n,s}(H) \gg 
 H^{(s-1)n^2 + w(n)/2 -n/2}.
 \]  
 \end{thm} 

As remarked above, we prove that  $\gamma_n>0$ for all $n\ge 2$ (and for $n=2$, our lower bound on $\gamma_n$ corresponds to Conjecture~\ref{conj:CharPoly}), see Theorems~\ref{thm:gamma23} and~\ref{thm:gamman}. 
In particular, for $n=2$ we have $\gamma_2 = 1$ 
and $w(2) = \varphi(3)^2 = 4$, hence  Theorem~\ref{thm:MultDep} gives matching lower and upper bounds
\[
 \# \cN_{2,s}(H) = H^{4s-3+o(1)}.
\]

We note that the upper bound in Theorem~\ref{thm:MultDep} is trivial for matrices $A_1,\ldots,A_s\in\SL_n(\Z)$. Indeed, if we denote 
$\cS_{n,s}(H)$ to be the set of  multiplicatively dependent 
$s$-tuples $(A_1, \ldots, A_s) \in \cM_n\(\Z; H\)^s$ with $\det A_i=1$, $i=1,\ldots,s$, then the trivial bound for $\#\cS_{n,s}(H)$ is $H^{sn^2-sn+o(1)}$ by Lemma~\ref{lem: Det} below. Here we can give a nontrivial bound only for $s=2$.

\begin{thm}
\label{thm:sln}
With $\gamma_n$ as in~\eqref{eq:CharPoly} and $w(n)$ as in~\eqref{eq:wn}, we have
\[
H^{2n^2-2n-\gamma_n+o(1)} \ge \#\cS_{n,2}(H)\gg H^{n^2-3/2n+\omega(n)/2} .
\]
\end{thm}

We note that if Conjecture~\ref{conj:CharPoly} holds, that is, if $\gamma_n=n(n-1)/2$, then Theorem~\ref{thm:sln} gives matching bounds in some situations. For example, when $n=\varphi(k)$, is a value of the Euler function,  
we have $\omega(n)=n^2$, and thus in this case Theorem~\ref{thm:sln} gives 
\[
\#\cS_{n,2}(H)= H^{3n(n-1)/2+o(1)}.
\]

We now note that the lower bound on $ \# \cN_{n,s}(H)$ in Theorem~\ref{thm:MultDep}
can also serve as a lower bound on  $ \# \cK_{n,s}(H)$, so we  now only present an upper bound. 

\begin{thm}
\label{thm:Kernel}
We have 
\[
 \# \cK_{n,s}(H)\le H^{sn^2-n + o(1)}.
\]
\end{thm}

In particular, we see from  Theorem~\ref{thm:Kernel} that almost all choices 
of the matrices $A_1, \ldots, A_s \in  \cM_n\(\Z; H\)$ generate a 
group for which the kernel of the 
abelianisation map~\eqref{eq:AbMap} contains only words $A_{i_1}^{\pm 1} \cdots A_{i_L}^{\pm 1}$, where 
for every $i$, the exponents of $A_i$ sum up to zero.

\subsection{Counting  matrices which boundedly generate subgroups}  
Motivated by the notion of  bounded generation, see~\cite{CDRRZ, CRRZ} for further references,  
we also ask for an upper bound on the cardinality of the set $\cG_{n,s}(H)$ 
of $s$-tuples $(A_1, \ldots, A_s) \in \cM_{n}(\Z;H)$ of non-singular matrices such that   
\[
\{A_1^{k_1}\ldots A_s^{k_s}:~k_1, \ldots, k_s \in \Z\} = \langle A_1\rangle \ldots  \langle A_s\rangle \le\GL_n(\Q), 
\]
where, as usual, $\cH \le \cG$ means that $\cH$ is a subgroup of a group $\cG$ and  $\langle A\rangle$ 
denotes the cyclic group generated by $A$.

We remark that just the fact that $I_n$ belongs to every subgroup of $\GL_n(\Q)$ shows that we cannot use our bounds on 
$\# \cN_{n,s}(H)$ from Section~\ref{sec:md matr} since now the choice $k_1=\ldots = k_s =0$ is not excluded. 
Nevertheless, we complement the underlying argument with some other ideas and show that  $\cG_{n,s}(H)$ 
is a rather sparse set.  This result can be viewed as dual to the recent work 
of Corvaja,  Demeio,  Rapinchuk,  Ren and Zannier~\cite{CDRRZ} on sparsity 
of elements of boundedly generated subgroups of $\GL_n(\Q)$. Some sparsity results
can also be derived from the work of   Chambert-Loir and  Tschinkel~\cite{C-LT}. 

\begin{thm}
\label{thm:BD gen}
For $n \ge 2$, we have 
\[
\# \cG_{n,s}(H) \le  H^{sn^2 -  (s-1)n/3+o(1)}. 
\]
\end{thm}

\section{Analytic number theory background}

\subsection{Integers with  prime divisors from a prescribed set}
  
For an integer $Q\ne 0$ and a real $U\ge 1$, let $F(Q,U)$  be the number 
of positive integers $u \le U$, whose prime divisors are amongst those of $Q$.

As in~\cite[Section~3.1]{PSSS}, in particular, see the derivation of~\cite[Equation~(3.10)]{PSSS},  using a result of  de Bruijn~\cite[Theorem~1]{dBr}, one immediately derives 
the following result (for which we supply a short sketch of the proof).

\begin{lemma}
\label{lem: Sunits}
 For any integer $Q\ge 1$ and   real $U\ge 1$, we have 
\[
F(Q,U) = \(Q U\)^{o(1)} , \qquad \text{as} \ U \to \infty.
\] 
\end{lemma}

\begin{proof} Let $p_1 <  \ldots < p_\nu \le U$ be all prime divisors of $Q$ up to $U$ and let 
$q_1 <  \ldots < q_\nu$ be the first $\nu$ primes. Replacing each $p_i$ with $q_i$
in any integer which contributes to $F(Q,U)$ may only reduce the size of these 
integers, and distinct integers have distinct images under this transformation. 
Hence $F(Q,U) \le \psi(U, q_\nu)$, where $\psi(x,y)$ is the number of positive 
integers up to $x$ with prime divisors up to $y$. Since clearly $\nu!\le Q$, 
by the Stirling formula (or just by an elementary bound $\nu^\nu/\nu!\le
  e^\nu$)   and assuming that $Q$ is sufficiently large,  we have 
\[
\nu \le (1+ o(1))\frac{\log Q}{\log \log Q},
\]
and by the prime number theorem $q_\nu \le  (1+ o(1))\log Q$. Now,
using the bound on $\psi(x,y)$  due to   de Bruijn~\cite[Theorem~1]{dBr}, 
after simple calculations we obtain the result. 
\end{proof}

\subsection{Rational  points on varieties over finite fields}
\label{sec:rat point}

We need a version of the celebrated result of Lang and
Weil~\cite{LaWe}, but adapted to affine varieties, see, for
example,~\cite[Theorem~7.5]{CaMa}.  Let $q$ be a power of
  a prime number.

\begin{lemma}\label{lem:LaWe} Let $\cV$ be an  algebraic subset of
  affine $m$-space defined
  over a finite field $\F_q$ of $q$ elements. Let $r=\dim \cV>0$ and let $d$ be the sum of the
  degrees of all $\F_q$-irreducible components of $\cV$. Finally, let
  $\sigma\ge 0$ denote the number of $r$-dimensional $\F_q$-irreducible components of $\cV$  that
  are  absolutely irreducible.
  Then 
  \[
   \# \cV(\F_q) =\sigma q^r     +O(q^{r-1/2}),
  \]  where the implied constant depends only on $m$ and $d$.
\end{lemma}
\begin{proof}
  In the notation of~\cite[Theorem~7.5]{CaMa} we have $\delta= d$ and
  $\Delta\le d$. 
  The asymptotic estimate holds when $q>2(r+1)d^2$. If $q\le
  2(r+1)d^2$ then $q\le 2(m+1)d^2$. The estimate also holds as $\#\cV(\F_q)\le q^m$ and $\sigma
  q^r \le dq^m$. 
\end{proof}

\section{Preliminaries on matrices}
\subsection{Matrices with a fixed determinant}
 
We need a bound on the number of matrices $A \in \cM_n\(\Z;H\)$ with prescribed value 
of the determinant $\det A =d$. We recall that  Duke,  Rudnick and  Sarnak~\cite{DRS},
if $d \ne 0$, and  Katznelson~\cite{Kat2} when $d=0$, have obtained asymptotic formulas 
(with the main terms of orders $H^{n^2 - n}$ and $H^{n^2 - n}\log H$, respectively) 
for the number of  such matrices when $d$ is fixed. However this is too restrictive for our purpose 
and so we use a uniform with respect to $d$ upper bound which is a special case 
of~\cite[Theorem~4]{Shp}.

\begin{lemma}
\label{lem: Det}
For any integer $d$, there are at most  $O\(H^{n^2 - n}\log H\)$ matrices  
$A \in \cM_n\(\Z;H\)$ with $\det A =d$.
\end{lemma} 

\subsection{Pairs  of multiplicatively  dependent matrices} 
 
We deal with the case  $s=2$ separately. 
 
 \begin{lemma}
\label{lem: md pairs}
With $\gamma_n$ as in~\eqref{eq:CharPoly}, we have 
 \[\# \cN_{n,2}(H) \le H^{2n^2 - n - \gamma_n + o(1)}.
 \]
 \end{lemma}

\begin{proof}  Let $(A,B) \in \cN_{n,2}(H)$. Thus $A^k = B^m$ for some integers $k$ and $m$
with $(k,m) \ne (0,0)$. 

We now see that for every eigenvalue $\lambda$ of $A$  there is an eigenvalue
$\mu$ of $B$ such that 
\begin{equation}
\label{eq: md pair}
\lambda^k = \mu^m. 
\end{equation}
Clearly $\lambda$ and $\mu$ are of {\it logarithmic height\/}  $O\(\log H\)$, we refer to~\cite{BomGub}
for a background on heights.
It follows from a very special case  of a result  of Loxton and van der Poorten~\cite[Theorem~3]{LvdP}
(see also~\cite{LoMa,Matveev}) that if we have a relation~\eqref{eq: md pair} for some $k$ and $m$ 
with $(k,m) \ne (0,0)$, then we also have it for some $k,m = O(\log H)$
(which may depend on $\lambda$ and $\mu$),
where also, without loss of generality, we can now assume that $k\ne 0$.  
Hence each eigenvalue $\mu$ of $B$ gives rise to $O\((\log H)^2\)$ 
possible eigenvalues $\lambda$ of $A$. Hence, for each of $O\(H^{n^2}\)$ choices  
of $B \in  \cM_n(\Z;H)$  we have  $O((\log H)^{2n})$  possibilities for the characteristic 
polynomial of $A$ and thus by~\eqref{eq:CharPoly}, we have at most 
 $H^{n^2 - n - \gamma_n + o(1)}$  choices  for $A \in \cM_n(\Z;H)$. The result now follows. 
\end{proof}

As in the remark after Theorem~\ref{thm:sln}, we note that, if
Conjecture~\ref{conj:CharPoly} holds, that is, if $\gamma_n=n(n-1)/2$,
then the upper bound in Lemma~\ref{lem: md pairs} and the lower bound
in Theorem~\ref{thm:MultDep} (for $s=2$)  match in some
situations. For example, when $n=\varphi(k)$, is a value of the Euler function,  
we have $\omega(n)=n^2$, and thus in this case we have
\[
\#\cN_{n,2}(H)= H^{3n^2/2-n/2+o(1)}.
\]

\subsection{Matrices in the centraliser} 
Given $A$, we now estimate the cardinality of the set 
\[
\cC_n(A, H) = \{B \in \cM_n(\Z;H):~AB = BA\}
\]
of matrices $B \in \cM_n(\Z;H)$ which  commute with $A$. In other words, 
$\cC_n(A, H)$ is the 
set of matrices  $B \in \cM_n(\Z;H)$ which also belong to the 
{\it centraliser\/} of $A$. 

\begin{lemma}
\label{lem: Comm Matr}
Assume that $A$ has either a row or a column with two non-zero elements. 
Then 
\[
\#  \cC_n(A, H)  \ll H^{n^2-n}.
\]
\end{lemma}

\begin{proof} Without loss of generality we can assume that $A = (a_{ij})_{i,j=1}^n$ 
satisfies $a_{11} \ne 0$ and also $a_{1k} \ne 0$ for some $k  \ge 2$. 
We now consider matrices  $B = (b_{ij})_{i,j=1}^n \in \cM_n(\Z;H)$  with a fixed 
top row $\vb_1=(b_{11}, \ldots, b_{1n})$.  Examining the top row of each of the products $AB = BA$, 
we see that for every $j = 1, \ldots, n$ we have
\[
 \sum_{i=1}^n a_{1i} b_{ij} = \sum_{i=1}^n b_{1 i} a_{ij},  
\]
 and thus
\[
 \sum_{i=2}^n a_{1i} b_{ij} = \sum_{i=1}^n b_{1 i} a_{ij}  - a_{11} b_{1j}.
\]
Since the right hand side is fixed and  $a_{1k} \ne 0$ we see that there are 
$O\(H^{n-2}\)$ possible values for $(b_{2j}, \ldots, b_{nj})$ and hence 
$O\(H^{(n-2)n}\)$ possible values for the $(n-1)\times n$ matrix below the top row $\vb_1$ of $B$.
Since $\vb_1$ can take $O\(H^{n}\)$ possible values the result follows.
\end{proof} 

We remark that the identity matrix shows that the condition on $A$ in Lemma~\ref{lem: Comm Matr} is necessary.

\subsection{Matrices over finite fields with a given characteristic polynomial}
Let $q$ be a power of a prime number. We now denote by
$P_n(\F_q,f)$ the number of $n\times n$ matrices $A\in
\cM_n(\F_q)$ with characteristic polynomial $f$ over the finite field
$\F_q$ with $q$ elements.

We now recall  Reiner's result~\cite[Theorem~2]{Rei}.
It gives an explicit formula on the number of matrices with a given characteristic polynomial over $\F_q$.
However, for our purposes, we only need the following asymptotical formula for $P_n(\F_q,f)$, 
 which follows from this explicit formula.

\begin{lemma}\label{lem:Rei}  For any $n \ge 1$, uniformly over monic polynomials $f \in \F_q[X]$ of degree $n$, we have 
\[
P_n(\F_q,f) = \(1 + o(1)\) q^{n^2-n},  \qquad \text{as}\ q \to \infty.
\] 
\end{lemma}

\subsection{Properties of the variety of matrices with fixed characteristic polynomial}
Finally, we also need the following properties of the variety of matrices (over $\C$) 
with a given characteristic polynomial.
We referee to Appendix~\ref{app: deg Vf} for more precise geometry
results.  

\begin{lemma}
  \label{lem:IrredDimDeg}
  Let $n\ge 2$, $f\in \Z[X]$ monic of degree $n$ and $\cV_f \subseteq \C^{n\times n}$   the algebraic set of matrices with complex
  entries and  characteristic polynomial $f$. Then  $\cV_f$ is 
  an absolutely irreducible non-linear variety of dimension $\dim \cV_f = n^2 - n$. 
\end{lemma}   

\begin{proof}
  For brevity, we write $\cV$ for $\cV_f$. 
  Let $p$ be a prime number and 
  let $\overline f\in\F_p[X]$ denote the reduction of $f$ modulo $p$.
  We consider the algebraic subset $\overline{\cV}_{\F_p}$ defined over
  $\F_p$ of affine
  $n^2$-space consisting of $n\times n$ matrices with
  characteristic polynomial $\overline f$.

  A priori, the $\F_p$-irreducible components of $\overline{\cV}_{\F_p}$ may fail to be
  absolutely irreducible. 
  However, there exists a finite field extension $\F_q/\F_p$     
  such that all $\F_q$-irreducible components
  of $\overline{\cV}_{\F_q}$, which is $\overline{\cV}_{\F_p}$ considered as
  a variety defined over $\F_q$, are absolutely
  irreducible.

  All $\F_q$-irreducible components
  of 
  $\overline{\cV}_{\F_q}$  have dimension at least $n^2-n>0$ as $\overline{\cV}_{\F_q}$ is cut out by $n$ polynomial equations.
  Let $\sigma$ denote the
  number of $\F_q$-irreducible components of
  $\overline{\cV}_{\F_q}$ of dimension $r=\dim \overline{\cV}_{\F_q}$. Then $\sigma\ge 1$.

  On the one hand,  by Lemma~\ref{lem:LaWe} we have
  \[\#\overline{\cV}_{\F_q}(\F_q) =
  \sigma q^r + O(q^{r-1/2}).
  \]
  The error term
  depends on $n$ and the sum of the degrees of all $\F_q$-irreducible
  components  of $\overline{\cV}_{\F_q}$. However,
  the latter is bounded from above solely in terms of $n$ by a
  standard application of B\'ezout's Theorem. In particular, $\sigma$
  is bounded from above in terms of $n$.

  On the other hand,
  $\#\overline{\cV}_{\F_q}(\F_q) \le P_n(\F_q,\overline{f})$. So, by Lemma~\ref{lem:Rei} we have
  \[
  \#\overline{\cV}_{\F_q}(\F_q) \le
  (1+o(1))q^{n^2-n}.
  \]

  For all $p$ large enough in terms of $n$ we conclude that   $r\le n^2-n$
  and $\sigma=1$. Therefore,
  all $\F_q$-irreducible components of $\overline{\cV}_{\F_q}$ have dimension
  $n^2-n$ and thus  $\overline \cV_{\F_q}$ is
  absolutely irreducible of dimension $n^2-n$.

  Thus for all $p$ large enough,    $\overline \cV_{\F_p}$ is absolutely
  irreducible.  
  This implies that $\cV$, an affine variety defined over $\Q$, is absolutely irreducible.

We next show that $\cV$ is not affine linear. Assume it is, then this would imply that $\cV$ is closed under linear  combinations of the form $\mu A + (1-\mu) B$, $\mu \in \R$, of matrices $A,B \in \cV$.
 Let $f(X)=\prod_{i=1}^n(X-\lambda_i)$, where $\lambda_i\in \C$, $i=1,\ldots,n$. We define 
 \[
 A=\begin{pmatrix} \lambda_1 &1&0&\ldots &0\\ 0 & \lambda_2 &0&\ldots &0 \\ \vdots&\vdots &\vdots &\ldots & \vdots\\ 0  &0 &0& \ldots &\lambda_n\end{pmatrix}  \mand B=A^t.
\]
 We thus have that $\mu A +(1-\mu) B \in \cV$ for any $\mu\in\R$. Therefore, taking for example $\mu=1/2$, $\frac{1}{2} (A +B) \in \cV$ and thus it has characteristic polynomial $f$. However, simple computation shows that  in fact the characteristic polynomial of $\frac{1}{2} (A +B)$ is
 \[
 \det\(X\cdot I_n-\frac{1}{2} (A +B)\)=f(X) - \frac{1}{4}\prod_{i=3}^n(X-\lambda_i),
\]
 which is a contradiction. Therefore $\cV$ is not affine linear, which concludes the proof. 
 \end{proof}

\begin{rem} Lemma~\ref{lem:IrredDimDeg} shows  that $\deg \cV_f>1$. In fact, much more is true.  Let 
\[
D_f =  \frac {n!} {\mu_1! \cdots \mu_s!},
\]
where $\mu_1, \ldots, \mu_s$ are multiplicities of the roots of $f$. 
We observe 
 that $\cV_f$ has at least $D_f$ intersection points with a linear subspace 
 of dimension $n$   formed by diagonal matrices.
 To see this we simply put the roots 
 of the fixed characteristic polynomial $f$ on the main diagonal of such matrices (with corresponding multiplicities), 
 which can be done in exactly  $D_f$ ways. Hence, for the variety $\cV_f$ in Lemma~\ref{lem:IrredDimDeg},
 we have   $\deg \cV_f \ge D_f$. 
  In particular, if $f$ has no multiple roots, then  $\deg \cV_f = n!$. We show in Appendix~\ref{app: deg Vf} 
  that this holds for any $f$.  
  \end{rem}

  \section{Proof of Theorem~\ref{thm:gamma23}}  

\subsection{The case $n=2$}
In this case we count matrices
\[
A=\begin{pmatrix}  a_{11} &  a_{12}\\  a_{21} & a_{22} \end{pmatrix} \in \cM_2\(\Z; H\)
\]
satisfying
   \[
   \det A=a_{11}a_{22}-a_{12}a_{21}  =d \mand \Tr A =a_{11}+a_{22}=t,
   \]
 where $f(X) = X^2 - tX +d$ is the characteristic polynomial of $A$. 
 
Writing $a_{22} = t - a_{11}$,  this leads us to the equation    
   \[
   a_{11} \(t - a_{11}\)   - a_{12} a_{21} = d.
  \]
  Clearly, there are $O(1)$ values of $ a_{11}$ for which $D = 0$, 
  where $D = a_{11} \(t - a_{11}\)   -  d$. In this case $a_{12} a_{21} =0$ and
  we obtain $O(H)$ such matrices 
  $A \in \cM_2\(\Z; H\)$. 
For any of the remaining  $O(H)$ values of $ a_{11}$ for which $D \ne 0$,  
we recall   the  bound~\eqref{eq:tau}
and  obtain $\tau(D) = D^{o(1)} = H^{o(1)}$ possible values for $\(a_{12}, a_{21}\)$. Therefore, we have shown that $\gamma_2\ge 1$. 
Recalling the upper bound~\eqref{eq:gamma_n low} for $n=2$, we obtain that $\gamma_2=1$.

 \subsection{The case $n=3$} We note that some algebraic manipulations below were done using SageMath\cite{SageMath}. But they are simple
and elementary enough to be verified directly. 

Let $H\ge 0$ and
\begin{equation}
  \label{eq:charpoly3x3}
 f(X)=X^3+\ell_1  X^2+\ell_2    X+\ell_3     \in\Z[X]. 
\end{equation}
Let $\mathcal{R}_3(H;f)$ denote the set of matrices
\begin{equation}
  \label{def:A}
A = \begin{pmatrix} a_1& a_2 & a_3\\
  b_1& b_2 & b_3\\
  c_1& c_2 & c_3
\end{pmatrix}  
\end{equation}
with characteristic polynomial $f$  and 
entries  in $\Z\cap [-H,H]$.

The characteristic polynomial of $A$ is an
\textit{affine} linear map in the first
row of $A$. The associated linear map is
given by 
\[
  L_A\colon ~(a_1,a_2,a_3)\mapsto
  \det\begin{pmatrix} -a_1& -a_2 &
-a_3\\ -b_1& X-b_2 & -b_3\\ -c_1& -c_2 & X-c_3
  \end{pmatrix} = r_1 X^2 + r_2 X + r_3.
\]
We  represent and identify this map with the matrix
\[
  L_A = \begin{pmatrix}
    -1 & 0 & 0 \\
    b_2+c_3 & -b_1 & -c_1 \\
    b_3c_2-b_2c_3 & -b_3c_1+b_1c_3 & b_2c_1-b_1c_2
  \end{pmatrix},
\]
when working with the standard basis on $\Z^3$ and the basis $(X^2,X,1)$
of quadratic polynomials. 
Then
\[
  Q(A) = b_1b_2c_1 + b_3c_1^2 - b_1^2c_2 - b_1c_1c_3.
\]
is the determinant of $L_A$.  
Let us highlight that $Q(A)$ does not
depend on $a_1,a_2,a_3$.

If $Q(A)\ne 0$, then $A$ is the only matrix with characteristic
polynomial $f$ and second and third row equal to the corresponding
rows of $A$.
 
Generically, we express $a_2$
from the first row as a rational function in the remaining rows and in
$(\ell_1,\ell_2,\ell_3)$, see~\eqref{eq:QRdivisor} below. This imposes a  restriction when all quantities are
integers. Indeed,  one expects a rational and non-polynomial function to map integral
vectors to rational vectors.

\begin{lemma}
  \label{lem:Qzero1}
  We have
  \[
    \#\{ A \in \mathcal{R}_3(H;f):~Q(A)=0\text{ and }
    a_2b_1+a_3c_1\ne 0\} \le H^{4+o(1)}. 
  \]
\end{lemma}

\begin{proof}
  Fix $A$ in $\mathcal{R}_3(H;f)$.
  As $A$ has characteristic polynomial $f$ we consider the
  quadratic and linear term to find
  \[
    a_1+b_2+c_3 = -\ell_1   \quad\text{and}\quad
    -a_2b_1 + a_1b_2 - a_3c_1 - b_3c_2 + a_1c_3 + b_2c_3 = \ell_2   .
  \]
  We eliminate $a_1$ from the second equation and find that
  \begin{equation}
    \label{eq:abac=L}
     a_2b_1+a_3c_1= - \cL_A,
  \end{equation}
where
  \begin{equation}
    \label{eq:cond1}
 \cL_A= b_2^2 + b_3c_2 + b_2c_3 + c_3^2 + b_2\ell_1   + c_3\ell_1  
    + \ell_2 . 
  \end{equation} 

  Suppose $Q(A)=0$ and $a_2b_1+a_3c_1\ne 0$. In particular,
  $(b_1,c_1)\ne (0,0)$ and  $\cL_A \ne 0$ . Then 
  $(b_2,b_3,c_2,c_3)$ lies in the rank $3$ lattice
\[
    \Lambda_{(b_1,c_1)} = \left\{(x,y,z,t) \in\Z^4 :~  x (b_1c_1) + y c_1^2 + z (-b_1^2) + t (-b_1c_1)=0\right\}.
\]

  Let $\gcd(b_1,c_1)=r$.
  We thus have $(b_1,c_1)=r(u_1,v_1)$, where $u_1,v_1$ are coprime integers.
  Note that  $\Lambda_{(b_1,c_1)}$  equals
  $\Lambda_{(u_1,v_1)}$. Moreover, its
  determinant is
  \[\(\(u_1v_1\)^2 + v_1^4 + u_1^4 + \(u_1v_1\)^2\)^{1/2} =
  u_1^2+v_1^2.
  \]
  We count the number of lattice points in the box $[-H,H]^4$ using
  a result of 
  Schmidt~\cite[Lemma~2]{Schm}.
  Indeed, $[-H,H]^4$ is contained in  the Euclidean ball of radius $H$
  centred at the origin.
  In the notation of~\cite{Schm}
 we have $d(\Lambda^{k(|-i)})\ge 1$ since
  $\Lambda_{(u_1,v_1)}\subseteq\Z^4$.
  So we get
  \begin{equation}
    \label{eq:schmidtest}
   \#\( \Lambda_{(u_1,v_1)}\cap [-H,H]^4\) =
  O(H^3/({u}_1^2+{v}_1^2) + H^2).
  \end{equation}

  Let $N(H)$ denote the cardinality of the set of matrices in
  the statement of the lemma.
  We count these matrices by putting them into the following form
  \[
    A = \begin{pmatrix}-b_2-c_3-\ell_1  & a_2 & a_3\\
      ru_1& b_2 & b_3\\
      rv_1& c_2 & c_3
    \end{pmatrix}  
  \]
  with $\gcd(u_1,v_1)=1$ and  $r\ge 1$ in $\Z$.
  Then
  $r \mid \cL_A$, where $\cL_A$ is given by~\eqref{eq:cond1}.
  We find
  \begin{align*}
    N(H) &     \le \sum_{\substack{(u_1,v_1)\in[-H,H]^2\\  
        \gcd(u_1,v_1)=1}}
    \sum_{\substack{(b_2,b_3,c_2,c_3)\in \\
        \Lambda_{(u_1,v_1)}\cap[-H,H]^4}}
    \sum_{\substack{ r\mid  \cL_A }}
    \sum_{\substack{(a_2,a_3)\in[-H,H]^2\\  
    a_2 r u_1+a_3r v_1=-\cL_A}} 1.
  \end{align*}

  Let $(a_2,a_3),(a^*_2,a^*_3)$ appear as indices in the 
  fourth, that is, inner-most, sum.
  Then $(a_2-a^*_2,a_3-a^*_3) = \lambda (v_1,-u_1)$ for some
  $\lambda\in \Z$. We have $|\lambda|\le 2H/ \max\{|u_1|,|v_1|\}$.
  Therefore, the number of terms in the fourth sum is at most 
  $O(H/\max\{|u_1|,|v_1|\})$.

  Recall that $\cL_A \ne 0$. We have $\ell_1   \ll H$ and
  $\ell_2   \ll H^2$, so  $\cL_A  \ll H^2$. Thus, by~\eqref{eq:tau},  the third sum 
  is over at most $H^{o(1)}$  values of $r$.  
  
  Combining the  observations for the fourth and third sum yields
\[
    N(H)\le   \sum_{\substack{(u_1,v_1)\in[-H,H]^2\\  
        \gcd(u_1,v_1)=1}}
         \frac{H^{1+o(1)}}{\max\{|u_1|,|v_1|\}}  \#\( \Lambda_{(u_1,v_1)}\cap [-H,H]^4\) . 
\]
 
 Using~\eqref{eq:schmidtest},  we  find
\[
    N(H) \le \sum_{\substack{(u_1,v_1)\in[-H,H]^2\\  
            \gcd(u_1,v_1)=1}} \left(\frac{H^{4+o(1)} }{\max\{|u_1|,|v_1|\}^3}
      +\frac{H^{3+o(1)} }{\max\{|u_1|,|v_1|\}}\right).
\]
 Replacing the condition
  $\gcd(u_1,v_1)=1$ by $(u_1,v_1)\ne (0,0)$ and using that  
  \[
    \sum_{\substack{(u_1,v_1)\in[-H,H]^2\\ %%  \cap\Z^2\\ 
        (u_1,v_1)\ne (0,0)}}
    \frac{1}{\max\{|u_1|,|v_1|\}^\kappa}
    \ll 
      \begin{cases}
        H &\text{if } \kappa = 1, \\
        1 &\text{if }  \kappa = 3,
      \end{cases}
  \]
  we complete the proof. 
  \end{proof}

\begin{lemma}
  \label{lem:Qzero2}
  We have
  \[
    \#\{ A \in \mathcal{R}_3(H;f) :~Q(A)=0, \ c_1\ne 0\text{
      and }
    a_2b_1+a_3c_1=0\} \le H^{3+o(1)}.
  \]
\end{lemma}

\begin{proof}
   Suppose $A$ is any integer $3\times 3$ matrix as in~\eqref{def:A} with characteristic
  polynomial $f$ as in~\eqref{eq:charpoly3x3} and such that $Q(A)=0$, $c_1\ne 0$ and
  $a_2b_1+a_3c_1=0$.
    
  In the equation $\det A = -\ell_3    $ we eliminate $a_1$ using
  $-\ell_1  =a_1+b_2+c_3$. Thus
  \begin{align*}
    -a_3b_2c_1 + a_2b_3c_1 & + a_3b_1c_2 + b_2b_3c_2 - a_2b_1c_3 \\
    & - b_2^2c_3 +  b_3c_2c_3 - b_2c_3^2 + b_3c_2\ell_1   - b_2c_3\ell_1   + \ell_3    =0.
  \end{align*}
  We further substitute $a_3=-a_2b_1/c_1$ (which is possible since $c_1 \ne 0$), and multiply
  with $c_1$ to find
  \begin{align*}
    b_2b_3c_1c_2 - b_2^2c_1c_3 & + b_3c_1c_2c_3 - b_2c_1c_3^2 + b_3c_1c_2\ell_1  
    - b_2c_1c_3\ell_1     
    \\
    &  + c_1\ell_3 +(b_1b_2c_1 + b_3c_1^2 - b_1^2c_2 - b_1c_1c_3)a_2 = 0.
  \end{align*}
  The factor in front of $a_2$ equals $Q(A)=0$. Removing this term, we can then cancel 
  $c_1\ne 0$ and get
  \begin{equation}
    \label{eq:Qzeroeq1}
    b_2b_3c_2 - b_2^2c_3 + b_3c_2c_3 - b_2c_3^2 + b_3c_2\ell_1  
    - b_2c_3\ell_1   + \ell_3    =0.
  \end{equation}
  
  We note that~\eqref{eq:Qzeroeq1} can be rewritten as
    \begin{equation}
    \label{eq:Qzeroeq1.1}
    (b_2+c_3+\ell_1)(b_3c_2-b_2c_3)=-\ell_3.
  \end{equation}

  \textbf{Case 1:} We have $\ell_3\ne 0$.

  Since $\ell_3$ is fixed and of size $O(H^3)$, by~\eqref{eq:tau} we see that 
  both $b_2+c_3+\ell_1$ and $b_3c_2-b_2c_3$ take each  $H^{o(1)}$
  possible values. For each such value of $b_2+c_3+\ell_1$ and
  $b_3c_2-b_2c_3$, there are $O(H)$ possible pairs $(b_2,c_3)$ (coming
  from $b_2+c_3+\ell_1$ being fixed). For each such fixed tuple,
  by~\eqref{eq:tau} again, there are $H^{o(1)}$ possibilities for each
  of $b_3$ and $c_2$ if $b_3c_2\not=0$. Therefore, in this case there are
  $H^{1+o(1)}$ possible tuples $(b_2,b_3,c_2,c_3)$ with
  $b_3c_2\not=0$.

  What if $b_3c_2=0$? Then $b_2c_3$ takes $H^{o(1)}$
  possible values. We write $c_3 = x-\ell_1-b_2$ for one of at most
  $H^{o(1)}$ possible $x$. So $b_2(b_2+\ell_1-x)$ takes at most
  $H^{o(1)}$ values and the same follows for $b_2$. There are at most
  $H^{1+o(1)}$ possibly $(b_2,b_3,c_2,c_3)$.
  
  Moreover, since  $a_2b_1+a_3c_1=0$, we also have at most $H^{2+o(1)}$ possible tuples $(a_2,a_3,b_1,c_1)$. 
  
  Since $a_1$ is fixed by the trace, putting  together the above bounds we obtain $H^{3+o(1)}$ possible matrices $A$.
  
  \textbf{Case 2:} We have $\ell_3= 0$.
  
  Using now the identity~\eqref{eq:abac=L} in the proof of Lemma~\ref{lem:Qzero1} and $a_2b_1+a_3c_1=0$, we find 
  that $\cL_A = 0$. That is, 
  \begin{equation}
    \label{eq:Qzeroeq2}
    b_2^2 + b_3c_2 + b_2c_3 + c_3^2 + b_2\ell_1   + c_3\ell_1  
    + \ell_2   =0.
  \end{equation}
  From~\eqref{eq:Qzeroeq1.1} we obtain that 
  \begin{equation}
    \label{eq:Case 2b}
    b_3c_2-b_2c_3=0
  \end{equation}
  or
  \begin{equation}
    \label{eq:Case 2a}
    b_2+c_3+\ell_1=0.
  \end{equation}

\textbf{Subcase 2a:} We have  $b_3c_2-b_2c_3=0$, so~\eqref{eq:Case
  2b}.

Replacing $b_3c_2=b_2c_3$ in~\eqref{eq:Qzeroeq2}, we obtain
\begin{equation}
  \label{eq:b2c3quad}
  (b_2+c_3)^2 + \ell_1(b_2+c_3)+\ell_2=0,
\end{equation}
which implies that $b_2+c_3$ takes at most two values.

  First, suppose that $b_3c_2\not=0$. We
  obtain $O(H)$ possibilities for the pairs $(b_2,c_3)$. For each such
  pair, since $b_3c_2=b_2c_3$ is of size $O(H^2)$,
  by~\eqref{eq:tau} we have $H^{o(1)}$ possibilities for each of $b_3$
  and $c_2$. This leaves us with $H^{1+o(1)}$ possibilities for
  $(b_2,b_3,c_2,c_3)$.

  Second, say $b_3c_2=0$. Then $b_2c_3=0$.
  By~\eqref{eq:b2c3quad} there are at most 2
  possibilities for $b_2+c_3$. So at most four possibilites for
  $(b_2,c_3)$. Thus $O(H)$ possibilities for   $(b_2,b_3,c_2,c_3)$.

\textbf{Subcase 2b:} We have $b_2+c_3+\ell_1=0$, so~\eqref{eq:Case 2a}
holds.

Now there is a single possibility for $b_2+c_3$.

There are at most $O(H)$ possible tuples $(b_2,c_3)$. We fix one such tuple. 
Therefore, from~\eqref{eq:Qzeroeq2},  $b_3c_2$ is fixed of size
$O(H^2)$. Suppose $b_3c_2\not=0$, by~\eqref{eq:tau} there are thus again
 $H^{o(1)}$
possibilities for each of $b_3$ and $c_2$.

What if $b_3c_2=0$? We insert $c_3=-b_2-\ell_1$ from~\eqref{eq:Case 2a} 
into~\eqref{eq:Qzeroeq2} to find $b_2^2+b_2\ell_1+\ell_2=0$. Thus there are
at most 2 possibilities for $(b_2,c_3)$ and at most $O(H)$
possibilities for $(b_2,b_3,c_2,c_3)$. 

Thus, in each subcase  there are  in total $H^{1+o(1)}$ possible tuples $(b_2,b_3,c_2,c_3)$. Since, as in Case~1, using that $a_2b_1+a_3c_1=0$, there at most $H^{2+o(1)}$ possible tuples $(a_2,a_3,b_1,c_1)$, we conclude the proof.
\end{proof}
 
Suppose $A$ as in~\eqref{def:A} has characteristic polynomial $f$ as
in~\eqref{eq:charpoly3x3}. A
calculation shows that 
\begin{equation}
  \label{eq:QRdivisor}
  \(2b_2c_1c_2 - b_1c_2^2 + c_1c_2c_3 + c_1c_2\ell_1    + c_1^2 a_2\)Q(A) = R(A)
\end{equation}
with 
\begin{equation}
  \label{def:R}
    R(A) = -c_1^3f((b_2c_1-b_1c_2)/c_1);  
\end{equation}
the right-hand side is well-defined even if $c_1$ vanishes as $f\in\Z[X]$ has
degree $3$. 
The expression for $R(A)$ is independent of $b_3$ and $c_3$.
We make the following  useful observations: if $A$ is an integer matrix, then

\begin{itemize}
\item $Q(A)\in\Z$ and $R(A)\in \Z$;
\item if $Q(A)\ne 0$, then $Q(A)\mid R(A)$.
\end{itemize}

As observed above, $Q(A)$ depends only on the bottom two rows of
$A$. If $Q(A)\ne 0$ the top row
$(a_1,a_2,a_3)$ is
uniquely determined by $f$ and the bottom two rows.

\begin{lemma}
  \label{lem:RAzero}
  We have 
\[\#\{ A \in \mathcal{R}_{3}(H;f) :~Q(A)\ne 0, \ c_1\ne 0\text{ and }R(A)=0\} \le H^{4+o(1)}.
\]
\end{lemma}

\begin{proof}
  Let $A$ be from the above set. 
  Then $(b_2c_1-b_1c_2)/c_1$ is one of the at most three eigenvalues
  of $A$
  by~\eqref{def:R} and $R(A)=0$.
  We split our set  
  into two subsets  depending on whether $b_1c_2\ne 0$ or $b_1c_2=0$.

  \textbf{Case 1:}   Matrices with $b_1c_2\ne 0$.
  
  We sum over the $O(H^4)$ possibilities of $(b_2,b_3,c_1,c_3)$ among
  the entries of $A$.
  For such a quadruple there
  are at most $3$ possibilities for $b_1c_2$ as $(b_2c_1-b_1c_2)/c_1$
  is a root of $f$. This product is non-zero
  and $|b_1c_2|\le H^2$.  By~\eqref{eq:tau}
  there are at most $H^{o(1)}$ possible pairs $(b_1,c_2)$. Specifying $(b_1,c_2)$ leaves at most one
  possibility for the first row $(a_1,a_2,a_3)$ since $Q(A)\ne 0$.
  This leaves
  $H^{4+o(1)}$ matrices with $b_1c_2\ne0$.

  \textbf{Case 2:}  Matrices with $b_1c_2=0$.
  
  Then $(b_2c_1-b_1c_2)/c_1=b_2$ is
  a root of  the characteristic polynomial $f$. So either
  $(b_1,b_2)$ or $(b_2,c_2)$ comes from a set of at most $3$ elements.
  The count for the bottom two rows is $O(H^4)$. But they fix the top
  row and this completes the second case.
\end{proof}

\begin{lemma}
  \label{lem:c1nonzero}
  We have 
 \[
 \#\{ A\in  \mathcal{R}_3(H;f) :~c_1\ne 0\} \le H^{4+o(1)}.
 \]
\end{lemma}

\begin{proof}
  We separate our counting in three classes.
  
  \textbf{Case 1:} Matrices with $Q(A)=0$.

  The number of these matrices  is at most
  $H^{4+o(1)}$ by combining Lemmas~\ref{lem:Qzero1} and~\ref{lem:Qzero2}.

  \textbf{Case 2:} Matrices with $Q(A)\ne 0$ and $R(A)=0$.

  In this case, we apply Lemma~\ref{lem:RAzero}.

  \textbf{Case 3:} Matrices $A\in\mathcal{R}_3(H;f)$ with $Q(A)\ne 0$
  and $R(A)\ne 0$.
  
  In this case, the
  bottom two rows uniquely determine $(a_1,a_2,a_3)$. 
  We have $Q(A)\mid R(A)$ by~\eqref{eq:QRdivisor}.

  Note that $\ell_i \ll H^i$ for all $i= 1,2,3$. So $R(A) \ll H^6$ by~\eqref{def:R}. 
  Recall that $R(A)$ only involves $b_1,b_2,c_1,c_2$ and $f$.   
  For fixed $(b_1,b_2,c_1,c_2)$ the number
  of possible divisors of the non-zero $R(A)$ is $H^{o(1)}$.
  For such a quadruple,  by~\eqref{eq:tau},  the number of possible values of 
  $Q(A)$ is $H^{o(1)}$.

  Let $v \mid R(A)$ be such a divisor.
  Let us keep $(b_1,b_2,c_1,c_2)$ fixed.  Then $Q(A) = v$ translates to
  $c_1(c_1 b_3  - b_1 c_3)=w$ in $(b_3,c_3)$, where   $w$ is a function
  of $v$, and thus there are at most  $H^{o(1)}$
  values $w$. As $c_1\ne 0$ and $(b_1,b_2,c_1,c_2)$ are fixed, the
  equation $Q(A)=v$ determines a line in $b_3,c_3$.

  Furthermore, any two points $(b_3,c_3)$ and $(u_3,v_3)$ on this line
  satisfy
  \[
    c_1 (b_3-u_3) - b_1 (c_3-v_3)=0.
  \]
  So 
  \[
  (u_3,v_3) = (b_3,c_3) + \frac{\lambda}{\gcd(b_1,c_1)}(b_1,c_1)
\]
  for some $\lambda\in\Z$.
  Therefore, for fixed $w$,  the number of points $(b_3,c_3)$
  on this line and in  $[-H,H]^2\cap\Z^2$ is $O\(H \gcd(b_1,c_1)/\max\{|b_1|,|c_1|\}\)$.

  Our total point count in Case~3 with $c_1\ne 0$ is thus
  \begin{align*}
    \sum_{\substack{(b_1,b_2,c_1,c_2) \in[-H,H]^4\\  
    c_1\ne 0}}
    H^{1+o(1)} & \frac{\gcd(b_1,c_1)}{\max\{|b_1|,|c_1|\}}\\
    & \le H^{3+o(1)}
    \sum_{\substack{ (b_1,c_1)\in[-H,H]^2\\  
    (b_1,c_1)\ne \{(0,0)\}}}
    \frac{\gcd(b_1,c_1)}{\max\{|b_1|,|c_1|\}}.
  \end{align*}
We now easily derive
  \begin{align*}
    \sum_{r=1}^H 
    \sum_{\substack{ (b_1,c_1)\in[-H,H]^2\\  
    \gcd(b_1,c_1)=r}} &
    \frac{1}{\max\{|b_1/r|,|c_1/r|\}}\\
   & \le \sum_{r=1}^H \sum_{\substack{(b_1,c_1)\in [-H/r,H/r]^2\\ 
      \\ 
      (b_1,c_1)\ne (0,0)}}
    \frac{1}{\max\{|b_1|,|c_1|\}} \\
    &\ll \sum_{r=1}^H H/r = H^{1+o(1)}.
  \end{align*}
 So the contribution  from matrices in Case~3 is at most $H^{4+o(1)}$.
\end{proof}

We are now able to conclude the proof of Theorem~\ref{thm:gamma23} for
$n=3$. 
  Let $A\in \mathcal{R}_3(H;f)$.
  If $c_1$, the lower left entry of $A$, is non-zero, then $A$ is covered by
  the count in Lemma~\ref{lem:c1nonzero}.

  If we swap the second and
  third rows and columns of $A$, the characteristic polynomial and
  norm of $A$  stay the same. This operation swaps $b_1$ and $c_1$. So if
   $b_1\ne 0$ we can again appeal to 
  Lemma~\ref{lem:c1nonzero}.

  Therefore, it suffices to count $A\in \mathcal{R}_3(H;f)$ with $b_1=c_1=0$.
  In this case, $f$ is divisible by $X-a_1$. In particular, there are
  at most three values for $a_1$. We count the number of
  $(b_2,b_3,c_2,c_3)$ by applying Theorem~\ref{thm:gamma23}
  in the case $n=2$.    Their number is
  $H^{1+o(1)}$.
  Consider the remaining entries $(a_2,a_3)$ as free. Then the
  number of
  possibilities for $A$ is at most $H^{3+o(1)}$.

  \section{Proof of Theorem~\ref{thm:gamman}}  

\subsection{Determinant polynomials of some linear maps associated with characteristic polynomials} 
For $H\ge 2$
let $\mathcal{R}_n(H;f)$ denote the set of $n\times n$-matrices
with characteristic polynomial $f\in\Z[X]$ and 
coefficients in $[-H,H]\cap \Z$. We assume throughout this section that $n\ge 2$.

We use some ideas from the case $n = 3$, however for notational convenience 
in our definition of the polynomials $Q(A)$  instead of the bottom $n-1$ rows,  we work with the 
first  $n-1$ columns.

More precisely, given an $n\times(n-1)$ matrix $A$ and a column vector $\va$ of  
dimension $n$, we write  $(A\vert \va)$ for the $n\times n$ matrix obtained by
augmenting $A$ by $\va$. 
Let $I^*_n$ denote the first $n-1$ columns of $I_n$.

Consider $n(n-1)$ independent variables
$a_{ij}$, $1\le i\le n, 1\le j\le n-1$, of the
polynomial ring $\Z[a_{ij}:~1\le i\le n, \ 1\le j\le n-1]$. 

Let for the moment $A$ denote the $n\times(n-1)$ matrix $A=(a_{ij})_{\substack{ 1\le i\le n \\ 1\le j\le n-1}}$.
As for the case $n=3$, we define a linear map
\[
  L_A :~ 
  \va=
  \begin{pmatrix}
    a_1 \\ \vdots \\a_n
  \end{pmatrix}
  \mapsto
    \begin{pmatrix}
    r_1(\va) \\ \vdots \\ r_n(\va)
  \end{pmatrix}
\] 
via 
\begin{align*}
r_1(\va)X^{n-1} +\cdots & + r_n(\va) = \det\bigl( X I^*_n -A \vert -\va)\\
& = \det
\begin{pmatrix}
  X-a_{1,1} & -a_{1,2} & \cdots & -a_{1}\\
  -a_{2,1} & X-a_{2,2} & \cdots & -a_{2}\\
  \vdots &\ddots &  & \\
  -a_{n-1,1} &\cdots & X-a_{n-1,n-1} & -a_{n-1}\\  
  -a_{n,1}&\cdots & -a_{n,n-1} & -a_{n}
  \end{pmatrix}.
\end{align*}

It is convenient to identify the linear map $L_A$ with its
representation matrix with respect to the standard basis of $\Z^n$ and
$(X^{n-1},\ldots,X,1)$.

As usual, given a set $\cS$, we use $ \Mat_{m,n}(\cS)$
to denote the set of all $m \times n$ matrices with entries from the set $\cS$.
We also write  $ \Mat_n(\cS) = \Mat_{n,n}(\cS)$.

Thus $L_A\in \Mat_n(\Z[a_{ij}:~1\le i\le n, \ 1\le j\le n-1])$.  
The coefficients in
the $k$-th row of $L_A$ are the coefficients of
$r_k$ seen as a linear form in $\va$.

For example, $r_1(\va) = -a_n$ and so the first row of $L_A$ is
$(0,\ldots,0,-1)$. 
Moreover, $r_n(\va) = (-1)^n \det (A\vert \va)$ (in particular, the $n$-th row of $L_A$ can be determined by
the Laplace expansion of $\det(A\vert \va)$ by the last column).

The characteristic polynomial of $(A\vert \va)$
equals
\begin{equation}
  \label{eq:charpolyformula}
  \begin{split}
  \det\(X I_n-\left(A|\va\right)\)&\\
  =
r_1(\va)X^{n-1}&+\cdots + r_n(\va) + \det\({X I^*_n-A}|{X \ve_n}\), 
\end{split} 
\end{equation}
where $\ve_n = (0,\ldots,0,1)^t$. 

The determinant of $L_A$ is a polynomial 
\[
Q\in\Z[a_{ij}:~1\le i\le n, \ 1\le j\le n-1]
\] 
in the $n(n-1)$ variables $a_{ij}$, $1 \le i \le n$, $1 \le j \le n-1$.

Now, let $A$ be an $n\times(n-1)$ matrix with entries  in any ring.
Then $Q(A)$ denotes the evaluation of $Q$ at the entries of $A$.
Let $\bar A= (A\vert \va)$ be an $n\times n$-matrix with
entries in a ring, with $A$ an $n\times(n-1)$ matrix and  $\va$  a
column vector of length $n$ as  above. It is useful to
define $Q(\bar A)$ to be $Q(A)$.

\begin{lemma}
  Let $K$ be a field, $f\in K[X]$ monic of degree $n$
  and $A\in \Mat_{n,n-1}(K)$. If $Q(A)\not=0$ then there exists a
  unique $\va\in K^n$ such that $(A\vert \va)$ has characteristic polynomial $f$.
\end{lemma}
\begin{proof}
  This follows from~\eqref{eq:charpolyformula} and since $Q(A)=\det L_A$.
\end{proof}

Next we show that $Q$ does not vanish identically on the set of  matrices of any
prescribed characteristic polynomial. In particular, this  implies that $Q$ is a non-zero polynomial. 

\begin{lemma}
  \label{lem:Qnonvanishing}
  Let $K$ be an algebraically closed field and 
  let $f\in K[X]$ be monic of degree $n$. There exists a matrix in
  $\Mat_n(K)$ with characteristic polynomial $f$ that is not
  in the vanishing locus of
  $Q$.
\end{lemma}
\begin{proof}
  Say $f(X) = (X-\lambda_1)\cdots (X-\lambda_n)$ with $\lambda_i\in K$, $i=1,\ldots,n$. We define
\[
    A=
    \begin{pmatrix}
      \lambda_1 & 0 &  \cdots & & 0 \\
      1& \lambda_2 & 0 & \cdots & 0 \\
      0& 1  & \lambda_3 & 0 &  \\
      \vdots&   & \ddots & \ddots &\vdots   \\        
      0 &\cdots  & & 1 & \lambda_{n-1} & \\
            0 &\cdots  & & 0 & 1 & 
    \end{pmatrix} \in \Mat_{n,n-1}(K).
\]
  Suppose $L_A(\va) =\mathbf 0$  (or $L_A \va =\mathbf 0$ if we recall that we identify $L_A$
  with a matrix) with $\va\in K^n$.
  In other words, $\det (X I^*_n-A|-\va)=0$. We claim that $\va=\mathbf{0}$.
  This claim is immediate when $n=2$ as then $\det( X I^*_2-A|-\va) = -a_2
  (X-\lambda_1)-a_1$.
  Suppose $n\ge 3$. 
  Then $\det\left(\lambda I^*_n -
      A|-\va\right)=0$ for all $\lambda\in K$. For
  $\lambda=\lambda_1$, the Laplace expansion by the last column yields
  $\det\left({\lambda_1 I^*_n -
      A}|{-\va}\right)=\pm a_1$. So $a_1=0$. We  expand by the first
  row and find
  \begin{align*}
   & r_1(\va)X^{n-1}+\cdots + r_n(\va)\\
    & \quad = (X-\lambda_1)\det  \begin{pmatrix}
      X-\lambda_2 & 0 &  \cdots & & 0 & -a_2 \\
      -1& X-\lambda_3 & 0 & \cdots & 0 & -a_3 \\
      \vdots& \ddots  & \ddots &  &\vdots &   \\        
      0 &\cdots  & & -1 & X-\lambda_{n-1} & -a_{n-1}\\
      0 &\cdots  & & 0 & -1 & -a_{n}\\
    \end{pmatrix}.
  \end{align*}
  
  As this expression is identically zero, we conclude $a_2=\cdots=a_n=0$ 
  by a simple induction.

  Thus $L_A$ has a trivial kernel. This means that $Q(A)\not=0$. The result now 
  follows as $(A\vert \va)$ with $\va
  = (0,\ldots,0,\lambda_n)^t$ has characteristic polynomial $f$.
\end{proof}

\begin{lemma}We have:
  \label{lem:LAproperties}
  \begin{enumerate}    
  \item[(i)] for all $k\in \{1,\ldots,n\}$,  each entry in the $k$-th row of
    $L_A$ is 
    homogenous of degree $k-1$ in the variables $a_{ij}$, $i =1, \ldots, n$,  $j = 1, \ldots, n-1$;
  \item[(i)]  the polynomial $Q$ is homogenous of degree $n(n-1)/2$ in the
  variables   $a_{ij}$, $i =1, \ldots, n$,  $j = 1, \ldots, n-1$.
  \end{enumerate}
\end{lemma}

\begin{proof}
  For Part~(i) we recall that 
  \[
  r_1(\va)X^{n-1}+\cdots + r_n(\va) =
  \det\left({X I^*_n -A}|{-\va}\right)
\] 
  with $A=(a_{ij})_{\substack{ 1\le i\le n \\ 1\le j\le n-1}}\in
  \Mat_{n,n-1}(\Z[a_{ij}:~1\le i\le n, \ 1\le j\le n-1])$,
  where   $\va$ is a column vector of
  length $n$. We multiply $X$ and $A$ by a
  new unknown $t$. This scales the right-hand side by $t^{n-1}$. Thus
  $r_k$, which defines the $k$-th row of $L_A$, is homogenous of degree $k-1$ in $a_{ij}$, $1\le i\le n$, $1\le j\le n-1$.

For Part~(ii) we recall that  the polynomial $Q$ is the determinant of $L_A$. 
  By Part~(i), we have that $Q$ is
  homogenous of degree 
  \[\sum_{i=1}^n \deg r_i = \sum_{i=1}^n (i-1) = n(n-1)/2,
\]
  which concludes the proof.
\end{proof}

Let $f(X)=X^n+\ell_1X^{n-1}+\cdots + \ell_n$ with new variables $\ell_1,\ldots,\ell_n$. 

Suppose $A=(a_{ij})_{\substack{ 1\le i\le n \\ 1\le j\le n-1}}$ is the $n\times(n-1)$ matrix whose entries
are 
variables $a_{ij}$ from above.
We consider a column vector $\va$ with coefficients in the fraction
field of $\Z[a_{ij},\ell_k:~ 1\le i,k\le n,\  1\le j\le n-1]$. We consider the equation
\begin{equation}
  \label{eq:LAaeqn}
  L_A \va =
  \begin{pmatrix}
    \ell_1-h_1 \\
    \vdots \\
    \ell_n - h_n
  \end{pmatrix},
\end{equation}
where 
\[
  \det(X I^*_n - A | X \ve_n) = X^n  + h_1 X^{n-1}+\cdots + h_{n-1}X
  + h_n. 
\]
Here each $h_k\in \Z[a_{ij}:~1\le i\le n, \ 1\le j\le n-1]$  is a homogenous polynomial of degree $k$.
The characteristic polynomial of $(A\vert \va)$ equals $f$
by~\eqref{eq:charpolyformula}.

The determinant $Q(A) = \det L_A \in \Z[a_{ij}:~1\le i\le n,\  1\le j\le n-1]$  is non-zero by Lemma~\ref{lem:Qnonvanishing}.
So
\[
  \va = L_A^{-1}   \begin{pmatrix}
    \ell_1-h_1 \\
    \vdots \\
    \ell_n - h_n
  \end{pmatrix} 
\]
is well-defined and provides the unique solution of~\eqref{eq:LAaeqn}. 

Let  $L_A^\#$  denote the adjugate matrix of $L_A$ and
 $\(L_A^{\#}\)_{i,j}$  the entry in row $i$
and column $j$ of $L_A^{\#}$. Up-to sign this entry
equals the determinant of $L_A$ after
deleting column $i$ and row $j$. By Lemma~\ref{lem:LAproperties}(i),
the element $(L_A^{\#})_{i,j}$ of $L_A$ at the position $(i,j)$,   is a
homogenous polynomial in $\Z[a_{ij}:~1\le i\le n, \ 1\le j\le n-1]$ of  degree 
\[
\left(\sum_{k=1}^n (k-1)\right)-(j-1)=n(n-1)/2-j+1.
\]

For $\nu\in \{1,\ldots,n\}$ we
define 
\[
Q_\nu\in \Z[a_{ij},\ell_k:~ 1\le i,k\le n, \ 1\le j\le n-1]
\] 
to be the $\nu$-th entry of
\begin{equation}
  \label{def:Qi}
  L_A^{\#}
  \begin{pmatrix}
    \ell_1-h_1 \\
    \vdots \\
    \ell_n - h_n
  \end{pmatrix}.
\end{equation}
So, with $\vell =\(\ell_1, \ldots, \ell_n\)$ we have 
\begin{equation}
  \label{eq:aiQiQ}
  a_1 = \frac{Q_1 (A, \vell)}{Q(A)},\ldots,a_n = \frac{Q_n (A, \vell)}{Q(A)}
\end{equation}
are the coordinates of the (unique) solution $\va$ of~\eqref{eq:LAaeqn}.
That is, the matrix
\[
  \begin{pmatrix}
    a_{1,1} & \cdots & a_{1,n-1} & \frac{Q_1 (A, \vell)}{Q(A)} \\
    \vdots & & \vdots & \vdots \\
    a_{n,1} & \cdots & a_{n,n-1} & \frac{Q_n (A, \vell)}{Q(A)}\\
  \end{pmatrix}
\]
has the characteristic polynomial equal to $f = X^n + \ell_1X^{n-1}+\cdots
+ \ell_n$.

The degree of $Q_\nu$ with respect to $a_{ij}$, $1\le i\le n$, $1\le j\le n-1$, satisfies
\begin{equation}
  \label{eq:degQibd}
  \begin{split}
  \deg_{(a_{ij})_{i,j}} Q_\nu& \le
  \max_{1\le j \le n} \deg\( (L_A^{\#})_{\nu,j} (\ell_j-h_j)\)\\
  & \le  \max_{1\le j \le n}
\(\frac{n(n-1)}{2}-j+1+j\) = \frac{n(n-1)}{2}+1.  
\end{split}
\end{equation}
The degree of $Q_\nu$ with respect to $(\ell_1,\ldots,\ell_n)$ is at most $1$. 

By computing the trace we find 
\[{Q_n}/{Q} = -\ell_1 - a_{1,1}-\cdots
-a_{n-1,n-1}.
\]
So $Q_n/Q$ is actually a polynomial. But not all
quotients $Q_\nu/Q$ are polynomials by the next  result.

\begin{lemma}
  \label{lem:notpoly}
  Specialize the variables  $\ell_1,\ldots,\ell_n$ to
  elements in $\Q$. There exists $k\in \{1,\ldots,n\}$ such that
  $Q_k/Q\not\in \Q[a_{ij}:~1\le i\le n,\  1\le j\le n-1]$.
\end{lemma}
\begin{proof}
  We suppose that $Q_k/Q = R_k$ is a polynomial in
  $\Q[a_{ij}:~1\le i\le n, \ 1\le j\le n-1]$ for all $k=1,\ldots,n$. 
  Then
  \begin{equation}
    \label{eq:imageA}
    \left( A \left|
      \begin{array}{c}
        R_1(A) \\ \vdots \\ R_n(A)
      \end{array}\right.
    \right)
  \end{equation}
  has characteristic polynomial $f = X^n + \ell_1X^{n-1}+\cdots + \ell_n$
  for all $A\in \Mat_{n,n-1}(\C)$.

  We have $\deg R_k = \deg Q_k - \deg Q\le
  \frac{n(n-1)}{2} + 1 -\frac{n(n-1)}{2}=1$ by Lemma~\ref{lem:LAproperties}(ii)
  and~\eqref{eq:degQibd}.
  Here the degree is understood to be relative to $(a_{ij})_{i,j}$.

  Recall that the
  set of matrices with prescribed characteristic polynomial is an
  absolutely irreducible variety of dimension $n(n-1)$, see Lemma~\ref{lem:IrredDimDeg}.
   The 
  matrices~\eqref{eq:imageA} constitute
  an $n(n-1)$-dimensional affine  linear subvariety of the $n^2$-dimensional affine space
  $\Mat_n$. From this we deduce
  that the set of matrices with prescribed characteristic polynomial
  is affine linear.  This  contradicts Lemma~\ref{lem:IrredDimDeg}. 
\end{proof}

For a  polynomial $Q$ with complex coefficients
(in our case $Q$ is a polynomial in variables $a_{ij}$, $1\le i\le n$, $1\le j\le n-1$), we let $|Q|$ denote
its naive height, that is, 
the maximal absolute value of its coefficients.

\begin{lemma}
  \label{eq:Qinorm}
  Specialize the variables  $\ell_1,\ldots,\ell_n$ to
  elements in $\Q$. We have $|Q|=O(1)$ and $|Q_i| = O(\max\{1,|\ell_1|,\ldots,|\ell_n|\})$, $i=1,\ldots,n$.
\end{lemma}
\begin{proof}
  The matrix $L_A$ is independent of $\ell_1,\ldots,\ell_n$. Therefore, so
  is $Q(A)=\det L_A$. 
Hence, the coefficients of $Q(A)$ depend only on $n$, 
  and in particular are $O(1)$.

By definition, each polynomial $Q_i$ appears in the $i$-th entry of~\eqref{def:Qi}. The polynomials
  $h_1,\ldots,h_n$ and the entries of $L_A^{\#}$
  are independent of $\ell_1,\ldots,\ell_n$. Moreover,
  $Q_i$ is linear in $(\ell_1,\ldots,\ell_n)$. Therefore, $|Q_i| =
  O(\max\{1,|\ell_1|,\ldots,|\ell_n|\})$. 
\end{proof}

\subsection{Concluding the proof} 
We assume that $H\ge 2$ as otherwise there is nothing to prove.

Let $f = X^n+\ell_1X^{n-1}+\cdots + \ell_n\in\Z[X]$ be given.

We keep the notation involving $Q,Q_1,\ldots,Q_n\in
\Z[a_{ij}:~1\le i\le n, \ 1\le j\le n-1]$, where the variables $\ell_1,\ldots,\ell_n$ are
specialized to the coefficients of $f$, in this case integers. There
is nothing to specialize in $Q$, which does not involve the $\ell_i$.

By Lemma~\ref{lem:notpoly} there is  $i_0$ such that $Q_{i_0}/Q$ is
\textit{not} a polynomial in rational coefficients. 
We may write $Q = C \widetilde Q$ and $Q_{i_0} = C \widetilde{Q}_{i_0}$
with $C$ and $\widetilde Q,\widetilde{Q}_{i_0}\in
\Z[a_{ij}:~1\le i\le n, \ 1\le j\le n-1]$ and the latter two coprime. As $Q_{i_0}/Q$ is not
a polynomial (with possibly rational coefficients), we see that
$\widetilde{Q}$ is non-constant and $\widetilde{Q}_{i_0}\not=0$.

Let us select a variable $z$, among the $a_{ij}$, that appears in
$\widetilde Q$. So
\begin{equation}
  \label{eq:Qz}
  \widetilde Q = \sum_{k=0}^{\deg_z \widetilde Q}
  \widehat{Q}_{k} z^{\deg_z\widetilde Q-k}
  \quad\text{with}\quad \deg_z \widetilde Q\ge 1,
\end{equation}
where
$\widehat{Q}_k\in \Z[a_{ij}:~1\le i\le n, \ 1\le j\le n-1]$ is independent of
 $z$ for all $k$ and where $\widehat Q_0\not=0$.

Next we eliminate the variable $z$ from $\widetilde Q$ and
$\widetilde{Q}_{i_0}$. 
Let $R$ be the resultant of $\widetilde{Q}$ and $\widetilde{Q}_{i_0}$ with
respect to $z$. 
 Then $R\in \Z[a_{ij}:~1\le i\le n,\  1\le j\le n-1]$ does not
depend on $z$. Moreover, by~\cite[Corollary~6.20]{GaGe}, since $\widetilde{Q}$ and $\widetilde{Q}_{i_0}$ are coprime  (and thus in particular as polynomials in $z$), $R$ is not zero, and by~\cite[Corollary~6.21]{GaGe}, there exist non-zero $U,V\in \Z[a_{ij}:~1\le i\le n,\ 1\le j\le n-1]$ with $\deg_z U < \deg_z \widetilde{Q}_{i_0}$ and $\deg_z V < \deg_z
\widetilde{Q}$ (with the standard convention that 
 the zero
  polynomial is of degree $-\infty$), such that
  \begin{equation}
  \label{eq:Requation}
 R = U \widetilde{Q} +V \widetilde{Q}_{i_0}.
\end{equation} 

We also have the following estimates:
\begin{lemma}
  \label{lem:deghgtRub}
  We have $\deg R \ll 1$ and $|R| \le \(|f|+1\)^{O(1)}$.
\end{lemma}

\begin{proof}
  By Lemma~\ref{eq:Qinorm} we have $|Q| \ll 1$ and
  $|Q_{i_0}| \ll |f|$.

  The polynomial $\widetilde Q$ is a divisor of $Q$ in
  $\Z[a_{ij}:~1\le i\le n, \ 1\le j\le n-1]$. We can bound $|\widetilde Q|$ from above using
  Gelfond's Lemma. Indeed,~\cite[Lemma~1.6.11]{BomGub} and $\deg Q \ll 1$
  imply $|\widetilde Q| \ll |Q|$. Hence
  $|\widetilde Q| = O(1)$ as $Q$ is independent of $f$.

  The same argument works for $\widetilde Q_{i_0}$, which is a divisor
  of $Q_{i_0}$. Indeed, now we use
  $\deg \widetilde Q_{i_0} \le \deg Q_{i_0}\le n(n-1)/2+1 \ll 1$ and
  $|Q_{i_0}|  \ll |f|$ by~\eqref{eq:degQibd} and
  Lemma~\ref{eq:Qinorm}, respectively. Hence $|\widetilde
  Q_{i_0}| \ll |f|$ by Gelfond's Lemma~\cite[Lemma~1.6.11]{BomGub}.

  Recall that $R$ is the resultant of $\widetilde Q$ and $\widetilde
  Q_{i_0}$ with respect to the variable $z$. The resultant is the
  determinant of the Sylvester matrix of these two polynomials.
  A short calculation yields $\deg R  \ll 1$ (this also follows from~\eqref{eq:Requation}). 
  We bound the norm of this determinant using a
  standard application of the triangle inequality. To bound the norm
  of a product of polynomials we use the upper bound of~\cite[Lemma~1.6.11]{BomGub}, which implies 
  $|R| = O(|f|^{O(1)})$.  
\end{proof}  

We now bound $\#\mathcal{R}_n(H;f)$ from above. We estimate the number of
elements in 
$\mathcal{R}_n(H;f)$ by splitting this set up into three subsets. 
In each case below, we  find at most $H^{n(n-1)-1+o(1)}$ matrices.

\textbf{Case 1:}  The set of matrices $(A\vert \va)\in \mathcal{R}_n(H;f)$ with  $Q(A)=0$.

Let $\mathcal{V}$ be the algebraic set of complex $n\times n$ matrices
with characteristic polynomial $f$.
Recall that $\mathcal{V}$ is an absolutely irreducible variety of
dimension $n(n-1)$, see Lemma~\ref{lem:IrredDimDeg}. 

By Lemma~\ref{lem:Qnonvanishing}, $Q$ does not
vanish identically on the set of algebraic points of $\mathcal{V}$.
Therefore, the intersection of $\mathcal{V}$ with the zero locus of
$Q$ is a possibly reducible variety  of dimension $n(n-1)-1$.
Hence, the
number of matrices $(A\vert \va)$ arising in Case~1 is $O(H^{n(n-1)-1})$, which 
is better than the desired bound.  Indeed, for linear components of  the above variety this is obvious, 
while for absolutely irreducible components of degree at least $2$
we can apply the bound of Pila~\cite[Theorem~A]{Pila} (alternatively, 
we can recall a more powerful~\cite[Lemma~3]{H-B}).

\textbf{Case 2:}  The set of matrices $(A\vert \va)\in \mathcal{R}_n(H;f)$ with  $Q(A)\not=0$ and $R(A)\widehat Q_0(A)=0$. 

Recall that $R\in \Z[a_{ij}:~1\le i\le n,\ 1\le j\le n-1]\smallsetminus\{0\}$ and that
$\deg R \ll 1$, see Lemma~\ref{lem:deghgtRub}. Moreover, $\deg
\widehat Q_0 \le \deg \widetilde Q\le \deg Q = n(n-1)/2$ by
Lemma~\ref{lem:LAproperties}(ii) and $\widehat Q_0\not=0$.
Trivially,  we have
\[\#\{ A\in \Mat_{n,n-1}(\Z;H) :~R(A)\widehat Q_0(A)=0\}  \ll H^{n(n-1)-1}.\]

In the current case we also assume $Q(A)\not=0$. This condition
implies that, given $A$, there is at most
one $\va\in\Z^n$ such that $(A\vert \va)$ has characteristic polynomial $f$.
Therefore, the total number of matrices in $\mathcal{R}_n(H;f)$ that
land in Case~2 is $O(H^{n(n-1)-1})$. As in Case~1, this is better than 
 the desired bound.

\textbf{Case 3:}   The set of matrices   $(A\vert \va)\in \mathcal{R}_n(H;f)$ with \[Q(A)R(A)\widehat Q_0(A)\not=0.\]

As in Case~2, $(A\vert \va)$ is uniquely determined by $A$ and the characteristic
polynomial.

We  have  $Q(A) = C(A)
\widetilde Q (A)$ and $C(A),\widetilde
Q(A)\in\Z\smallsetminus\{0\}$. Moreover, $\widetilde Q_{i_0}(A)\in\Z$.
By the discussion around~\eqref{eq:aiQiQ}, the $i_0$-th entry of the final column of $(A\vert \va)$ equals
$Q_{i_0}(A)/Q(A) = \widetilde Q_{i_0}(A)/\widetilde Q(A)$. This entry
is an integer. Therefore, $\widetilde Q(A)\mid \widetilde
Q_{i_0}(A)$. Now $R(A)\in \Z$ and~\eqref{eq:Requation} implies $\widetilde Q(A) \mid
R(A)$. 

We may assume that there is at least one matrix in $\mathcal{R}_n(H;f)$. Then
$|f|  \ll H^n$. 
So $\deg R  \ll 1$ and $|R|  = H^{O(1)}$ by
Lemma~\ref{lem:deghgtRub}.
The
triangle inequality implies $|R(A)| =H^{O(1)}$. 

We identify $\Mat_{n,n-1}(\Z)$ with $\Z^{n(n-1)}$ by arranging
entries from left-right, top-down. Under this identification let
\[ \pi\colon\Mat_{n,n-1}(\Z)\rightarrow\Z^{n(n-1)-1}\] denote the
projection onto those entries that do not correspond to the
distinguished coefficient $z\in \{a_{ij} :~1\le i\le n,\ 1\le j\le
n-1\}$. Given $A_z\in\Z^{n(n-1)-1}$ it makes sense to evaluate
$R(A_z)$ and all $\widehat Q_k(A_z)$ as $R$ and the $\widehat Q_k$ are
independent of $z$.

If $A_z=\pi(A)$ with $(A\vert \va)$ as in this case, then the discussion
above implies $\widetilde Q(A) \mid R(A_z)$. So the number of possible
divisors $\widetilde Q(A)$ is at most $H^{o(1)}$ for fixed $A_z$, 
see~\eqref{eq:tau}.

After these preliminary remarks, we proceed by counting the matrices in this third case.  Write 
\[ 
N(H) = \#\{ (A\vert \va)
\in \mathcal{R}_n(H;f) :~Q(A)R(A)\widehat Q_0(A)\not=0\}.
\] Then
\begin{equation}
  \label{eq:NHbound}
  N(H) \le \sum_{\substack{A_z \in \Z^{n(n-1)-1} \\ \|A_z\|_\infty \le
      H \\ R(A_z)\widehat Q_0(A_z)\not=0}}
  \#\{(A\vert \va) \in \mathcal{R}_n(H;f):~\pi(A)=A_z\},
\end{equation}
where $\|A_z\|_\infty$ is the largest absolute values of all entries
of $A_z$.

Given $A_z$ as in the  sum we estimate the number of possible
$(A\vert \va)\in\mathcal{R}_n(H;f)$  with $\pi(A)=A_z$. 
There is  $\delta\in \Z\smallsetminus\{0\}$ among one of the $H^{o(1)}$
divisors of $R(A_z)$ with $\widetilde Q(A)=\delta$.
By~\eqref{eq:Qz}, the entry of $A$ missing in $A_z$ is a root of
\[
\sum_{k=0}^{\deg_z \widetilde Q} \widehat Q_k(A_z) z^{\deg_z\widetilde
  Q-k} - \delta
  \in\Z[z].
\]   
As we are counting matrices satisfying  $\widehat Q_0(A)=\widehat Q_0(A_z)\not=0$ this
polynomial has degree $\deg_z \widetilde Q \ge 1$, see~\eqref{eq:Qz}. In particular, it
is not identically zero. Its number of integral roots is at most
$\deg_z \widetilde
Q \le \deg \widetilde Q \le \deg Q  \ll 1$. 

Therefore, for a given $A_z$ and $\delta$, both as above, the number of possible $(A\vert \va)$ with $\pi(A)=A_z$ is $O(1)$. Since the number of possible $\delta\in \Z\smallsetminus\{0\}$ is $H^{o(1)}$, we conclude $N(H) = H^{n(n-1)-1+o(1)}$ from~\eqref{eq:NHbound}.

This completes the third case and concludes  the proof of
Theorem~\ref{thm:gamman}.

  \section{Proof of Theorem~\ref{thm:MaxRank}}
  
 \subsection{Upper bound} 
 \label{sec:Up Bound}
Using the multiplicativity of the determinant  and since we only count multiplicatively depended $s$-tuples of 
non-singular matrices, we can deduce from~\eqref{eq:MultDep}
that
\begin{equation}
\label{eq:Partition}
\prod_{i \in \cI} \(\det A_i\)^{|k_i|} = \prod_{j \in \cJ} \(\det A_j\)^{|k_j|} \ne 0,  
\end{equation}
where the sets $\cI, \cJ \subseteq\{1, \ldots, s\}$ do not intersect.
By the maximality of the rank condition, we also have $\cI \cup \cJ= \{1, \ldots, s\}$ and
also $|k_h| > 0$, $h =1, \ldots, s$. 

We now fix some sets $\cI$ and $\cJ$ as above 
and count $s$-tuples for which~\eqref{eq:Partition} is possible with these sets 
$\cI$ and $\cJ$  and some nonzero exponents $|k_h| > 0$, $h =1, \ldots, s$. 
Let $I = \# \cI$ and $J = \# \cJ$. 

Assume that $J  \le I$. We then fix $J$ matrices $A_j$, $j \in \cJ$, which is possible in at most 
\[
\fA_1 \ll H^{Jn^2}
\]
ways. Let
\[
Q =  \prod_{j \in \cJ} |\det A_j|.
\]
Clearly the determinants $\det A_i$, $i \in \cI$, are factored from the prime divisors of $Q$ and thus, using the trivial bound on their size, we see that 
each of them can take at most $F(Q, n!H^n)$ values.
Trivially, $Q  \ll H^{Jn}$, and therefore  by Lemma~\ref{lem: Sunits} we see that 
$F(Q, n!H^n) = H^{o(1)}$.  For each of these values, we apply Lemma~\ref{lem: Det}, 
and see that each of the matrices $A_i$ can take at most $H^{n^2 - n+o(1)}$ values. 
Hence the total number of choices for the $I$-tuple $(A_i)_{i \in \cI}$ is at most 
\[
\fA_2 = H^{In^2 -In+o(1)}.
\] 

Hence, the total number of $s$-tuples $(A_1, \ldots, A_s) \in \cM_n\(\Z; H\)^s$
satisfying~\eqref{eq:Partition} for at least one choice of the exponents  is at most 
\[
\fA_1 \fA_2 = H^{Jn^2 +In^2 -In+o(1)} = H^{ sn^2  -In+o(1)}. 
\]
Since we have $O(1)$ choices for the sets $\cI$ and $\cJ$ and $I\ge J$ implies $I \ge \rf{s/2}$ 
(because $I+J=s$), we obtain the desired upper bound.

\begin{rem}  The above argument follows the idea of the proof of~\cite[Proposition~1.3]{PSSS} 
(with $n=1$), however the desired upper bound does not follow from the results of~\cite{PSSS}. 
Hence we supply a self-contained proof  and exhibit some 
ideas which are also used in other proofs. 
\end{rem} 
 
   \subsection{Lower bound}

The same argument as  in the proof of the upper bound, 
based on Lemmas~\ref{lem: Det} and~\ref{lem: Sunits}, 
 implies that   for every $j =2, \ldots, s$ there are at most 
$K^{sn^2  -n+o(1)}$  choices of $s$-tuples $(B_1, \ldots, B_s)  \in \cM_n\(\Z; K\)^s$ of non-singular matrices 
such that  all prime divisors of  $\det B_j$   divide  $\det B_1 \ldots \det B_{j-1}$.   
Indeed, when non-singular matrices $B_1, \ldots, B_{j-1}$ are fixed, 
by Lemma~\ref{lem: Sunits},  the determinant 
$\det B_j$  can take at most
\[
F\(|\det B_1 \ldots \det B_{j-1}|,n!K^{n}\)= K^{o(1)}
\] 
values  such that  all prime divisors of  $\det B_j$   divide  $\det B_1 \ldots \det B_{j-1}$, and thus by Lemma~\ref{lem: Det}  there 
at most $K^{n^2  -n+o(1)}$  such  choices for $B_j$.

Hence, discarding the above  $K^{sn^2  -n+o(1)}$ choices of  $(B_1, \ldots B_s)  \in \cM_n\(\Z; K\)^s$, we see that there are still $K^{sn^2  +o(1)}$ choices for  $s$-tuples $(B_1, \ldots, B_s)  \in \cM_n\(\Z; K\)^s$ of non-singular matrices 
such that   for every $j =2, \ldots, s$,  $\det B_j$ contains a prime divisor which 
does not  divide $\det B_1 \ldots \det B_{j-1}$.

Now, if $s=2r$ is even, we set $K = \fl{(H/n)^{1/2}}$. Next, 
for any choice of 
$(B_1, \ldots, B_s)  \in \cM_n\(\Z; K\)^s$ of  non-singular  matrices 
we define
\begin{equation}
\label{eq:Constr even}
A_{2i-1} = B_{2i-1} B_{2i},  \qquad   A_{2i} =   B_{2i+1}B_{2i}  , \qquad i =1, \ldots, r, 
\end{equation}
where we also set $B_{2r+1} = B_{s+1} = B_1$. 
Clearly 
\[
A_1 A_2^{-1} \ldots A_{2r-1} A_{2r}^{-1} = I_n.
\]
Hence the $s$-tuple $(A_1,\ldots,A_s)$ is multiplicatively dependent. 
Moreover, all matrices $A_i$ lie in $\cM_n\(\Z; H\)$. 

Now, we choose matrices $(B_1, \ldots, B_s)  \in \cM_n\(\Z; K\)^s$ such that  for every $j =1, \ldots, s$, the determinant $\det B_j$  has a prime divisor that does not divide  $\det B_1 \ldots \det B_{j-1}$. 
As we have seen, we have  $K^{sn^2  +o(1)}$ choices for  $(B_1, \ldots, B_s)  \in \cM_n\(\Z; K\)^s$. In this case, any  sub-tuple 
$\(A_{i_1}, \ldots, A_{i_t}\)$ of length $t < s$ with $1 \le i_1 < \ldots < i_t \le s$ 
satisfies no relation of the form 
\[
A_{i_1}^{m_1}\ldots  A_{i_t} ^{m_t} = I_n
\]
(where we can assume $m_t\ne 0$) since  the matrix $B_{i_t+1}$,
with a determinant containing a unique prime divisor,  appears   exactly once 
in a non-zero power. Hence the matrices in~\eqref{eq:Constr even} are  multiplicatively dependent 
of maximal rank.

However, it is important to note that for different 
$s$-tuples of matrices $(B_1, \ldots, B_s)  \in \cM_n\(\Z; K\)^s$ the corresponding matrices  in~\eqref{eq:Constr even} 
can coincide. Hence we now need to estimate the possible frequency of repetitions. 
Clearly~\eqref{eq:Constr even} implies
\[
\det A_1= \det B_1 \det B_2.
\]
Since $\det A_1$ and $ \det B_1$ are non-zero integers of size at most $n!H^{n}$
the classical bound on the divisor function~\eqref{eq:tau}, implies that 
for a given $A_1$ we have at most $H^{o(1)}$ values for $ \det B_1$ and thus
 $K^{n^2 - n+o(1)}$ possible values of $B_1$ which for some $B_2,  \ldots, B_s  \in \cM_n\(\Z; K\)$
can lead to the same $s$-tuple  $(A_1, \ldots, A_s)$. Furthermore, when $A_1, \ldots, A_s$
and $B_1$ are fixed, the matrices $B_2,  \ldots, B_s$ are uniquely defined. 
Hence each  $s$-tuple  $(A_1, \ldots, A_s)$ comes from at most $K^{n^2 - n+o(1)}$
different choices of $(B_1, \ldots, B_s)  \in \cM_n\(\Z; K\)^s$, and thus
\begin{equation}
\label{eq:bound even}
 \# \cN_{n,s}^*(H) \ge K^{(s-1)n^2 +n +o(1)} =H^{(s-1)n^2/2 +n/2 + o(1)}
\end{equation}
for an even $s$.

If $s =2r+1$ is odd, we again set $K = \fl{(H/n)^{1/2}}$ and for any choice of 
$(B_1, \ldots, B_{s-1})  \in \cM_n\(\Z; K\)^{s-1}$ of   non-singular  matrices 
we define  
\begin{equation}
\label{eq:Constr odd}
\begin{split} 
A_{2i-1} = B_{2i-2} B_{2i-1}  , \quad A_{2i} &=   B_{2i}B_{2i-1},\quad i =1, \ldots r, \\
A_{2r+1}&=B_{2r},
\end{split}
\end{equation}
where we also set $B_0=I_n$. 
Clearly 
\[
A_1 A_2^{-1} \ldots A_{2r}^{-1} A_{2r+1} = I_n.
\]
Hence the $s$-tuple $(A_1,\ldots,A_s)$ is multiplicatively dependent
and each entry lies in $\cM_n\(\Z; H\)$. 

As above, we choose matrices $(B_1, \ldots, B_{s-1})  \in \cM_n\(\Z;
K\)^s$ such that  for every $j =1, \ldots, s$, the determinant $\det
B_j$  has a prime divisor that does not divide  $\det B_1 \ldots \det
B_{j-1}$. As in the case when $s$ is even, the associated
$s$-tuple $(A_1,\ldots,A_s)$ is multiplicatively dependent of maximal
rank. 
There are $K^{(s-1)n^2  +o(1)}$ possible $(s-1)$-tuples $(B_1, \ldots,
B_{s-1})$, and in this case  all of them generate distinct $s$-tuples $(A_1,\ldots,A_s)$ as in~\eqref{eq:Constr odd}. Therefore, we obtain
\begin{equation}
\label{eq:bound odd}
 \# \cN_{n,s}^*(H) \ge K^{(s-1)n^2  +o(1)} \ge H^{(s-1)n^2/2+o(1)}
\end{equation}
for an odd $s$.  

Putting~\eqref{eq:bound even} and~\eqref{eq:bound odd} together, we conclude the proof.

 \section{Proof of Theorem~\ref{thm:MultDep}}  
 \subsection{Upper bound} We define  the rank  of $(A_1, \ldots, A_s) \in 
 \cM_n\(\Z; H\)$, 
 as the largest integer $r$ with $1\le r\le s$ for which any sub-tuple 
\[
\(A_{i_1}, \ldots, A_{i_{r}}\), \qquad 1 \le i_1 < \ldots < i_{r} \le s,
\] 
is multiplicatively independent.
In particular,  rank $r =s$ means the 
multiplicative independence of  $(A_1, \ldots, A_s)$.   Hence we only count 
$s$-tuples with $r \le s-1$.

By Theorem~\ref{thm:MaxRank},  used with $r+1$ instead of $s$,  we see that  there are at most 
\[
H^{ (r+1) n^2  - \rf{(r+1)/2} n+o(1)}  H^{(s-r-1)n^2} = 
H^{ sn^2  - \rf{(r+1)/2} n+o(1)}  \le   H^{ sn^2  - 2 n+o(1)} 
\] 
choices 
of $(A_1, \ldots, A_s) \in  \cM_n\(\Z; H\)$ of rank $r $ with $n-1 \ge r\ge 2$.  

We now count    $s$-tuples $(A_1, \ldots, A_s) \in \cM_n\(\Z; H\)$ of rank $r=0$ and $r=1$.

Clearly for $r=0$, there is a matrix $A\in \{A_1, \ldots, A_s\}$ and an integer $k\ne 0$ with 
$A^k = I_n$. The same argument as  in the proof of  Lemma~\ref{lem: md pairs} implies that 
all eigenvalues of $A$ are roots of unity. Since they are of degree at most $n$ over $\Q$, there are
at most $O(1)$ possibilities for them and thus for the characteristic polynomial of $A$. 
Invoking  the definition of $\gamma_n$   in~\eqref{eq:CharPoly}, we obtain at most  
$H^{n^2-n -\gamma_n +o(1)}$ 
possible choices for $A \in \cM_n\(\Z; H\)$ thus  at most 
\[
H^{ (s-1)n^2 + n^2  - n -\gamma_n+o(1)} = H^{ sn^2  - n -\gamma_n+o(1)}
\] 
choices 
for  $(A_1, \ldots, A_s) \in  \cM_n\(\Z; H\)$.

For $r=1$, there are two distinct multiplicatively dependent matrices $A,B\in \{A_1, \ldots, A_s\}$ 
for which, by   Lemma~\ref{lem: md pairs}, there are at most $H^{2n^2-n -\gamma_n +o(1)}$ choices and 
thus    at most 
\[
 H^{ (s-2)n^2 + 2n^2  - n -\gamma_n+o(1)} = H^{ sn^2  - n -\gamma_n+o(1)}
\] 
choices 
for  $(A_1, \ldots, A_s) \in  \cM_n\(\Z; H\)$.

Putting everything together, we conclude the proof.

\subsection{Lower bound}  
We recall that $\widetilde R_{n} (H;f)$ is a natural analogue of 
$R_{n} (H;f)$, where  the matrices are ordered 
by the $L_2$-norm rather than by the $L_\infty$-norm.  In particular 
$R_{n} (H;f) \ge \widetilde R_{n} (H;f)$. 

Furthermore, let $m = \varphi(k)$ and  let $\Phi_k(X)$ be the $k$th cyclotomic polynomial.
Since $\Phi_k$ is a monic irreducible polynomial of degree $m$ over $\Z$, 
by~\eqref{eq:Asymp}  we see that for a sufficiently large $H$, there are 
\begin{equation}
\label{eq: low bound}
\begin{split} 
 R_{m} (H;\Phi_k) & \ge \widetilde R_{m} (H;\Phi_k) \\
 & = \(C\(\Phi_k\) + o(1)\) H^{m(m-1)/2}
 \gg H^{m(m-1)/2}
 \end{split} 
 \end{equation}
matrices $B \in  \cM_m\(\Z; H\)$ for which $\Phi_k(X)$ is the characteristic 
polynomial. Since $\Phi_k(X)\mid X^k-1$, we conclude that 
\[
B^k = I_m.
\]

Next, let $n = m_1+\ldots + m_h$ be a representation of $n$ as a sum of totients
$m_i = \varphi(k_i)$, $i =1,\ldots, h$. 

We now consider matrices $A$ of the form
\begin{equation}
\label{eq: MatrRootofUnity}
A = \begin{pmatrix} B_1 &  &\mathbf{0}\\
\mathbf{0}&\ddots & \mathbf{0} \\
 \mathbf{0}&  & B_h
\end{pmatrix} , 
 \end{equation}
with diagonal cells formed by matrices $B_i \in  \cM_{m_i} \(\Z; H\)$ with $B_i^{k_i} = I_{m_i}$ and zeros everywhere else. 
Clearly for $k = k_1 \ldots  k_h$ we have
\[
A^k =  \begin{pmatrix} B_1^k &  &\mathbf{0}\\
\mathbf{0}&\ddots & \mathbf{0} \\
 \mathbf{0}&  & B_h^k
\end{pmatrix} 
=   \begin{pmatrix} \(B_1^{k_1}\)^{k/k_1}  &  &\mathbf{0}\\
\mathbf{0}&\ddots & \mathbf{0} \\
 \mathbf{0}&  & \(B_h^{k_h}\)^{k/k_h} 
\end{pmatrix} = I_n.
\] 
We also see from~\eqref{eq: low bound},  using matrices $B_i$ whose characteristic polynomial
is $\Phi_{k_i}$, $i =1, \ldots, h$, that there are at least  
\begin{equation}
\label{eq: CountMatrRootofUnity}
\begin{split} 
\prod_{i=1}^h & \(C\(\Phi_{k_i} \)+ o(1)\) H^{m_i(m_i-1)/2}\\
& \qquad \qquad \gg  \prod_{i=1}^h  H^{m_i^2/2-m_i/2} =
  H^{-n/2} \prod_{i=1}^h  H^{m_i^2/2}
 \end{split} 
 \end{equation}
such matrices $A$.  

Hence, choosing $A_1$ to be one of the above   matrices and 
choosing  arbitrary  non-singular matrices 
 $A_2, \ldots, A_n \in  \cM_n\(\Z; H\)$, 
we obtain 
 \[
  \# \cN_{n,s}(H)  \gg H^{(s-1)n^2-n/2}   \prod_{i=1}^h  H^{m_i^2/2} 
\]
which concludes the proof.

 \section{Proof of Theorem~\ref{thm:sln}} 
The proof of the upper bound follows line by line that of Lemma~\ref{lem: md pairs}, with the only modification that, since $\det B=1$, the number of matrices $B$ therein is $H^{n^2-n+o(1)}$ by Lemma~\ref{lem: Det} (in fact $O(H^{n^2-n})$ by~\cite[Example~1.6]{DRS}). 

For the lower bound,  we take $A_1$  as in~\eqref{eq:
  MatrRootofUnity}, so 
\[
  A_1 = \begin{pmatrix} %(-1)^{\varphi(k_1)}
    B_1 &  &\mathbf{0}\\
\mathbf{0}&\ddots & \mathbf{0} \\
 \mathbf{0}&  &  B_h %(-1)^{\varphi(k_h)}  B_h
\end{pmatrix}, 
\]
where without loss of generality 
we can assume that $k_i\not= 2$, $i=1, \ldots, h$ (since $\varphi(2) = \varphi(1)$ this does not change the 
definition of $w(n)$). Then, for the corresponding cyclotomic polynomials we have 
$\Phi_{k} (0) = 1$ for all $k\ge 2$ and $\varphi(k)$ is even for all $k\ge 3$. Thus the matrices $B_i$, $i=1,\ldots,h$, in~\eqref{eq: MatrRootofUnity}  belong to $\SL_n(\Z)$. Therefore,  $A_1 \in \SL_n(\Z)$ and the number of choices for $A_1$ is still 
given by~\eqref{eq: CountMatrRootofUnity}.  After this,  we can choose   an arbitrary 
%matrices $A_2, \ldots, A_s \in \SL_n(\Z)\cap \cM_n(\Z;H)$ 
 matrix $A_2 \in \SL_n(\Z)\cap \cM_n(\Z;H)$
and note that by~\cite[Example~1.6]{DRS} for each of them we have $(c_n+ o(1))H^{n^2-n}$ choices, with some constant $c_n> 0$, 
which depends only on $n$. 
The lower bound now follows.

  \section{Proof of Theorem~\ref{thm:Kernel}}

  Assume that   $(A_1, \ldots, A_s) \in \cK_{n,s}(H)$. 
  From the definition of a nontrivial word in the kernel, taking determinants, 
 we obtain a relation  
\[
\prod_{i=1}^s\(\det A_i\)^{k_i} = 1, 
\]
with some integers $k_1, \ldots, k_s$,  not all equal zero.
Without loss of generality we can assume 
that $k_1\ne 0$. Then all prime divisors of $\det A_1$ are among those of the product 
\[
Q = \prod_{i=2}^s \det A_i;
\]
with the convention that $Q=1$ if $s=1$. 
Since, trivially, $|\det A_i| \le n!H^n$, $i=1, \ldots, s$, by Lemma~\ref{lem: Sunits},
we see that $\det A_1$  can take at most $H^{o(1)}$ distinct values. 
Now an application of Lemma~\ref{lem: Det} concludes the proof.

\section{Proof of Theorem~\ref{thm:BD gen}}   

\subsection{Preliminary discussion}
Without loss of generality we may assume $s\ge 2$.

There are obviously $O\(H^n\)$ matrices  $A\in \cM_{n} \(\Z; H\)$ 
which do not satisfy the conditions of Lemma~\ref{lem: Comm Matr}
(that is, such that each row and column  has at most one non-zero element). 
By  Lemma~\ref{lem: Det} there are at most 
$ H^{n^2 - n +o(1)}$ matrices $A\in \cM_{n} \(\Z; H\)$  with 
  $\det A_j = \pm 1$. Let $\cE$ be the set of such exceptional matrices 
  of both kinds. For $n \ge 2$,  we have 
 \[
 \# \cE  \ll H^n + H^{n^2 - n +o(1)} = H^{n^2 - n +o(1)}.
 \]
 Hence there are at most 
 \[
 H^{(n^2 - n +o(1))s} \le H^{sn^2 -  (s-1)n/3 + o(1)}
\]
 of $s$-tuples $(A_1, \ldots, A_s) \in \cE^s$. 

Hence, we can now consider the case where at least one matrix   $A_h \not \in   \cE$ among 
 $(A_1, \ldots, A_s)$  and thus $A_h$ satisfies the conditions of Lemma~\ref{lem: Comm Matr} and  also $\det A_h \ne \pm 1$. We fix such a matrix $A_h\not\in\cE$ for some $h\in\{1,\ldots,s\}$ in $O\(H^{n^2}\)$ ways, and we now count $(s-1)$-tuples $(A_1, \ldots,A_{h-1}, A_{h+1}, \ldots, A_s) \in \cM_{n}(\Z;H)^{s-1}$
  of non-singular matrices,  which boundedly generate a  subgroup of $\GL_n(\Q)$. 

\subsection{Initial step}  
 We also  define the set 
 \[
 \cL_1 =  \{1, \ldots, s\}\setminus \{h\}.
 \]
 
 Choose the smallest positive integer $\ell_1 \ne h$ 
 (that is, $\ell_1=2$ if $h =1$ and $\ell_1= 1$ if $h \ge 2$). If $\ell_1< h$ we consider the products $A_h A_{\ell_1}$ 
 and the product $A_{\ell_1} A_h$ otherwise. 
 
 Assume first $\ell_1 < h$. 
 If $ \langle A_1\rangle \ldots  \langle A_s\rangle \le \GL_n(\Q)$ then we have
\[
A_h A_{\ell_1} = A_1^{k_1}\ldots A_s^{k_s}
\]
with some $k_1, \ldots, k_s \in \Z$. Discarding vanishing exponents, we write 
\begin{equation}
\label{eq:Set  I1}
A_h A_{\ell_1} = \prod_{i\in \cI_1} A_i^{k_i}
\end{equation}
with integers $k_i \ne 0$  
and some set $\cI_1 \subseteq  \{1, \ldots, n\}$. Note that we do not rule out the case $\cI_1 = \emptyset$, 
that is, $A_h A_{\ell_1}  = I_n$, in which case $A_{\ell_1} = A_h^{-1}$ is uniquely defined. 
We also set 
\[\cJ_1 = \(\cI_1\cup  \{\ell_1\}\)   \setminus \{h\}  .
\]

If the right hand side  of the equation~\eqref{eq:Set  I1} is   of the form $A_{\ell_1} A_h$, then by Lemma~\ref{lem: Comm Matr}
we have at most $\#  \cC_n(A_h, H)  \ll H^{n^2-n}$ choices for $A_{\ell_1}$.  We  also have $\cJ_1 = \{\ell_1\}$ in this case. 

Otherwise, taking determinants we obtain 
\begin{equation}
\label{eq:Det Prod}
\det A_h  \det A_{\ell_1} =  \prod_{i\in \cI_1}\( \det A_i\)^{k_i}.
\end{equation}
Since $A_h \in \cM_{n}(\Z;H)$ is fixed, and the right hand side is not of the form $A_{\ell_1} A_h$ 
we obtain a nontrivial multiplicative relation among the determinants. 

We now use a modification of the argument of Section~\ref{sec:Up Bound}.

If $\cI_1 \subseteq \{h, \ell_1\}$ but $(k_h,k_{\ell_1}) \ne (1,1)$ then Lemma~\ref{lem: Det} 
implies that there are $H^{n^2 - n +o(1)}$ choices for $A_{\ell_1}$ (and as before 
we have  $\cJ_1 = \{\ell_1\}$). Indeed, since $A_h\not\in\cE$, we have  $\det A_h\ne \pm 1$, 
and thus the choice $k_h \ne 1$ and $k_{\ell_1} = 1$ is impossible.  So, we obtain 
\[
\(\det A_h\)^{k_h-1} = \(\det A_{\ell_1}\)^{1-k_{\ell_1}} 
\]
and thus all prime divisors of $\det A_{\ell_1}$ are among the  prime divisors of $\det A_h$. 
Therefore, by Lemma~\ref{lem: Sunits} there are $H^{o(1)}$ possibilities for  $\det A_{\ell_1}$
and Lemma~\ref{lem: Det}  yields the desired claim.

So we can now assume that $\cI_1$ contains some element which is different 
from $h$ and $\ell_1$. In particular in this case  for $J_1 = \# \cJ_1$ we have
$J_1 \ge 2$ and thus we have 
\[
\rf{(J_1-1)/2} \ge J_1/3.
\]

Hence as in Section~\ref{sec:Up Bound}, using Lemmas~\ref{lem: Det} 
and~\ref{lem: Sunits}, we see that there are at most 
\[
H^{n^2 J_1 - \rf{(J_1-1)/2} n +o(1)}  \le H^{n^2 J_1 - nJ_1/3 +o(1)}  
\] 
choices for $(A_j)_{j \in \cJ_1}$.  We note that the saving in the exponent is  
%:
only $\rf{(J_1-1)/2} n$ rather than  $\rf{J_1/2} n$  since it is possible that $\ell_1 \in \cI_1$ 
and $k_{\ell_1} = 1$, in which case the relation~\eqref{eq:Det Prod} does not impose 
any restrictions on $A_{\ell_1}$.  

A similar argument also applies for $\ell_1 > h$ with a similarly defined set $\cJ_1$. 
 Hence in either case we obtain the bound $ H^{n^2 J_1 - nJ_1/3 +o(1)}$ on the 
number of choices for $(A_j)_{j \in \cJ_1}$.

\subsection{Recursive steps}  

We now define 
\[
\cL_2 = \cL_1 \setminus \cJ_1.
\]

If $\cL_2 = \emptyset$ we stop, otherwise we chose the smallest integer $\ell_2\in \cL_2$ 
and consider one of  the products $A_h A_{\ell_2}$  or $A_{\ell_2} A_h$, depending on whether 
$\ell_2 < h$ or $\ell_2 > h$. Again we consider 
\[
A_h A_{\ell_2} = \prod_{i\in \cI_2} A_i^{k_i}
\qquad \text{or} \qquad 
A_{\ell_2}  A_h = \prod_{i\in \cI_2} A_i^{k_i}, 
\]
with integers $k_i \ne 0$ and some set $\cI_2 \subseteq  \{1, \ldots, n\}$.

We now define
\[
\cJ_2 =  \(\cI_2  \cup  \{\ell_2\}\) \setminus \( \cJ_1  \cup \{h\}\), 
\]
that is, for all
  $j\in\cJ_2$ the matrix $A_j$ does not appear in~\eqref{eq:Set I1}. 
Arguing as before we see that 
there are at most 
\[
H^{n^2 J_2 - \rf{J_2/2} n +o(1)} \le H^{n^2 J_2 - nJ_2/3 +o(1)} 
\] choices for $(A_j)_{j \in \cJ_2} $
where $J_2 = \# \cJ_2$. 

Continuing this procedure and generating new relation and sets of newly participating matrices,  we eventually arrive to the case when
\[
\cL_{r+1} = \cL_r \setminus \cJ_r = \emptyset.
\]  
Hence non-overlapping sets $\cJ_1, \ldots, \cJ_r$ exhaust the whole set $\cL_1$ and thus 
for their cardinalities $J_1, \ldots, J_r$ we have 
\begin{equation}
\label{eq:Sum Jt}
J_1 + \ldots + J_r = s-1.
\end{equation}

Since there are $O\(H^{n^2}\)$ choices for $A_h$,  we see from~\eqref{eq:Sum Jt}   that the total number of choices 
for the $s$-tuple $(A_1, \ldots, A_s) \in \cM_{n}(\Z;H)$ is
\[
H^{n^2} \prod_{t=1}^r H^{n^2 J_t  - nJ_t/3 +o(1)}
 \le H^{sn^2 -   (s-1)n/3+o(1)}, 
 \]
 which concludes the proof. 
 
 \section{Comments and open questions}
\label{sec:com} 

It is easy to see that some of our arguments 
%% both bounds of Theorem~\ref{thm:MultDep} 
 also apply to the counting 
function of $s$-tuples $(A_1, \ldots, A_s) \in \cM_{n}(\Z;H)$ such that 
\[
A_1^{k_1}\ldots A_s^{k_s} = A_1^{m_1}\ldots A_s^{m_s} ,
\]
with two distinct integer vectors 
\[
\(k_1,\ldots, k_s\) \ne \(m_1,\ldots , m_s\).
\]
This is another natural question which deserves a separate study.

We do not believe that  any of the upper and lower bounds in Section~\ref{sec:md matr}
are  tight except for the lower bound of  Theorem~\ref{thm:MultDep}, which could be 
close to the truth. That is, we expect that most of the $s$-tuples
$(A_1, \ldots, A_s) \in \cN_{n,s}(H)$ consist of one matrix, say $A_1$, with $A_1^k = I_n$
and arbitrarily chosen other matrices $A_2, \ldots, A_n \in \cM_n\(\Z; H\)$.  

The problem of counting matrices with a given characteristic polynomial has a very 
natural dual question: counting the numbers $P_n(H)$ and $I_n(H)$ of
distinct characteristic polynomials 
and irreducible characteristic polynomials, 
generated by  matrices $A \in \cM_n\(\Z; H\)$. We note that the methods and results 
of~\cite{ALPS, MeSh} lead to some lower bounds on  $P_n(H)$ and $I_n(H)$, but they are 
not expected to be tight. 

It is also very important to find a different approach than appealing to the determinant argument, taking also into account other coefficients of the characteristic polynomial. Such  an argument 
can open a path to extending our results for   $\SL_n(\Z)$ matrices, 
when the determinant argument provides no useful information. 
However, in the case  of $\SL_2(\Z)$ matrices an essentially optimal result is 
given in~\cite{BOS}.

Symmetric and other special matrices are of interest too. 
For symmetric matrices, a version of Lemma~\ref{lem: Det} is given 
in~\cite[Theorem~5]{Shp} but the result is much poorer than in the case of arbitrary matrices.
It is possible that  some ideas of Eskin and  Katznelson~\cite{EskKat} can be adapted to this case.

The bound of Lemma~\ref{lem: Comm Matr}  also motivates  dual questions of estimating the cardinalities
of the set of commuting matrices
\[
\cC_n(H) = \{(A,B)  \in \cM_n\(\Z;H\)^2:~AB = BA\},
\]
and also of the set of commutators of bounded size, namely of the set
\begin{align*} 
\cZ_n(H) = \{C \in \cM_n(\Z;H):~C& = AB A^{-1} B^{-1}\\
& \text{for some  non-singular}\ A,B  \in \cM_n(\Z)\}.
\end{align*}

Clearly $\cZ_n(H) \subseteq \{C \in \cM_n(\Z;H):~\det C = 1\}$, so by~\cite{DRS} we have
$\# \cZ_n(H)\ll H^{n^2-n}$.  On the other hand, we note  that it is quite possible that 
\[
 \{AB A^{-1} B^{-1}:~\text{for some non-singular}\ A,B  \in \cM_n(\Z)\} \supseteq \SL_n(\Z),
 \]
 for $n \ge 3$, see~\cite[Problem~8.13]{Shal}. 
   We also note that  a classical result of Feit and Fine~\cite{FeFi}
on counting commuting matrices over finite fields, applied with a  prime $p$ 
%% satsifying  - Typo to correct
satisfying 
$2H  <  p \ll H$, 
 instantly implies that  $\#  \cC_n(H)  \ll H^{n^2+n}$, but we seek better bounds.

By a result  of Loxton and van der Poorten~\cite[Theorem~3]{LvdP}
(see also~\cite{LoMa,Matveev}), which actually plays an important part in the argument 
of~\cite{PSSS}, if  non-zero integers $a_1, \ldots, a_s\in [-H,H]$ are multiplicatively dependent 
then one can choose reasonably small exponents in the corresponding relation.  
Namely there 
are $k_1, \ldots, k_s \in \Z$ not all zero such that $k_1, \ldots,
k_s \ll (\log H)^{s-1}$,   and 
\[
a_1^{k_1} \ldots a_s^{k_s} = 1.
\]
The following example, communicated to the authors by David Masser, shows that even 
for $2\times 2$ matrices nothing of this kind can be true. Take an
integer $H\ge 1$ and set
\begin{equation}
\label{eq:Bad Exp}
A_1 = \begin{pmatrix} 1 & H-1\\
0& 1\end{pmatrix}\mand A_2 = \begin{pmatrix} 1 & H\\
0& 1\end{pmatrix}.
\end{equation}
Then we have
\[A_1^k A_2^\ell =   \begin{pmatrix} 1 & (H-1)k + \ell H\\
0& 1\end{pmatrix}.
\] 
Hence 
\[A_1^k A_2^\ell = I_2
\]
with $(k,l)\not=(0,0)$ implies $|k|\ge H$. 
Thus for $n \ge 2$ no logarithmic bound is possible. However we deal with a finite set of matrices from $ \cM_n(\Z;H)$, so there is certainly some bound on the smallest possible exponents, which depends only on $H$ and which is very interesting to find. 

Finally, we note that any further progress towards Conjecture~\ref{conj:CharPoly} beyond our Theorems~\ref{thm:gamma23} and~\ref{thm:gamman}    is of independent interest.

\appendix

\section{Dimension and degree of the projective variety of matrices with a given characteristic polynomial}
\label{app: deg Vf}

\subsection{Formulations of main results}   
In this appendix we show that the  degree of the  variety  of $n\times
n$-matrices
with prescribed integral characteristic polynomial   is $n!$. It is perhaps
well-known, but we have not been able to find any 
precise reference.
This computation is not used
 in our main results, but note Remark~\ref{rem:dim growth}. 

Let $K$ denote an algebraically closed field of characteristic $0$.
Let $n\ge 1$ be an integer and let as usual $I_n$ denote the $n\times n$ identity matrix.
We consider $A$ as an $n\times n$ matrix with independent entries  
 $A_{ij}$.
The characteristic polynomial $\det(XI_n - A)$ can be written as
\[
\det(XI_n - A)=X^n +a_1(A) X^{n-1} + \cdots + a_n(A),
\]  where
$a_k \in \Z[(A_{ij})_{1\le i,j\le n}]$ is homogeneous of
degree $k$
for each $k\in \{1,\ldots,n\}$.
We fix a further variable $Z$.

Let $\ell_1,\ldots,\ell_n\in  K$. Then
\begin{equation}
  \label{eq:system}
  a_1(A) - \ell_1 Z, a_2(A) - \ell_2 Z^2, \ldots, a_n(A)-\ell_n Z^n
\end{equation}
are $n$ homogenous polynomials of degree $1,2,\ldots,n$, respectively, in
the $n^2+1$ variables coming from the entries of $A$ and $Z$.

We consider the entries of $A$ together with $Z$ as projective
coordinates on $\P^{n^2}_K$. We identify affine $n^2$-space
$\A^{n^2}_K$ as the Zariski open subset  of $\P^{n^2}_K$ given by $Z\not=0$.

By basic algebraic geometry, the subscheme $\cV_{\ell_1,\ldots,\ell_n}\subset
\P_K^{n^2}$ cut out by~\eqref{eq:system} is non-empty. Moreover,
each irreducible component of $\cV_{\ell_1,\ldots,\ell_n}$ has dimension at least $
n^2-n$.   We consider
$\cV_{\ell_1,\ldots,\ell_n}$  as a scheme over
$\mathrm{Spec}(K)$. 

\begin{prop}
\label{prop:irred scheme}
   If $\ell_1,\ldots,\ell_n\in\Z$, the scheme $\cV_{\ell_1,\ldots,\ell_n}$ is irreducible of dimension $n^2-n$.
\end{prop}
\begin{proof}
  When $\ell_1,\ldots,\ell_n$ are integers
  and $K$ has characteristic $0$, then by Lemma~\ref{lem:IrredDimDeg}
  the affine part
  $\cV_{\ell_1,\ldots,\ell_n}\cap \A^{n^2}_K$ is irreducible of dimension
  $n^2-n$. The hyperplane at infinity,  that is, $Z=0$,
  corresponds via~\eqref{eq:system} to nilpotent matrices.
  By Lemma~\ref{lem:IrredDimDeg} it also has  dimension $n^2-n$ as
  an affine variety.
  But it is a cone above $\cV_{\ell_1,\ldots,\ell_n}\cap
  (\P^{n^2}_K\setminus\A^{n^2}_K)$, which thus
  has dimension $n^2-n-1$. It follows that there can
  be no additional irreducible components of $\cV_{\ell_1,\ldots,\ell_n}$ on
  $Z=0$. 
\end{proof}

\begin{rem}
We note that Lemma~\ref{lem:IrredDimDeg}, and thus
Proposition~\ref{prop:irred scheme}, can be  extended to hold for
$\ell_1,\ldots,\ell_n\in\overline{\Q}$.  
\end{rem}

As  $\cV_{\ell_1,\ldots,\ell_n}$ is cut out by polynomials of degree
$1,2,\ldots,n$,  B\'ezout's Theorem implies that the degree is at most
$n!$. In the proposition below we establish equality, its proof is
carried out in this section.

\begin{prop}
  \label{prop:degreenfac}
   If $\ell_1,\ldots,\ell_n\in\Z$, the projective subvariety $\cV_{\ell_1,\ldots,\ell_n}\subset\P_K^{n^2}$ has
  degree $n!$.
\end{prop}

We establish Proposition~\ref{prop:degreenfac} in Section~\ref{sec:Dim Vn}

\subsection{Preliminaries} 
We call a matrix in $\Mat_n(K)$ Jordan regular if the Jordan blocks
in its Jordan normal form 
have pairwise distinct eigenvalues. For example, a matrix whose
characteristic polynomial is squarefree is Jordan regular.

\begin{lemma}
\label{lem:lin indep}
  Let $B\in \Mat_n(K)$ be Jordan regular. Then
  \begin{equation*}
    \left(\frac{\partial a_1 }{\partial A_{ij}}(B)\right)_{1\le i,j\le
      n},
    \ldots,     \left(\frac{\partial a_n }{\partial A_{ij}}(B)\right)_{1\le i,j\le n}
  \end{equation*}
  are linearly independent elements in $\Mat_n(K)$. 
\end{lemma}
\begin{proof}
  We use  Jacobi's formula to formally compute the derivative of $\det(X I_n
  -A)$ by $A_{ij}$. Indeed, we obtain
  \begin{equation*}
    \frac{\partial}{\partial A_{ij}} \det (X I_n-A) = - \mathrm{Tr}((X
    I_n-A)^{\#} E_{ij}), 
  \end{equation*}
  where $E_{ij}$ is zero except for a $1$ in row $i$, column $j$
  and %% the superscript 
 $B^\#$ denotes the adjugate matrix of  an $n \times n$ matrix $B$,  that is, 
\[
B B^\#= \det B \cdot I_{n}\,.
\] 

  The trace above is the $j$-th row and $i$-th column entry of
  $(XI_n-A)^{\#}$.  
  We get
  \begin{equation*}
   \left(\frac{\partial}{\partial A_{ij}} \det (X I_n-A)\right)_{1\le i,j\le n}
    = -((X I_n-A)^{\#})^t = -(XI_n - A^t)^{\#}.
  \end{equation*} 
  The entries on the left-hand side are the derivatives of the characteristic
  polynomial by the $A_{ij}$. We find
  \begin{equation}
    \label{eq:derivcharpoly}
    \sum_{k=1}^n X^{n-k}
    \left(\frac{\partial a_k(A)}{\partial A_{ij}} \right)_{1\le i,j\le n}
    = \det(XI_n-A^t)\cdot (A^t-XI_n)^{-1}. 
  \end{equation} 

  Suppose that the $n$ matrices in the lemma are not linearly
  independent. So they all lie in a vector subspace
  $V\subset\Mat_{n}(K)$ with $\dim V<n$.
  Moreover, any $K$-linear combination lies in $V$.
  We substitute $B$ in~\eqref{eq:derivcharpoly}. The left-hand side
  then lies in $V$ if we replace $X$ by any element of $K$.
  Thus 
  $(B^t- zI_n)^{-1} \in V$ for all $z\in K$ that are not
  eigenvalues of $B^t$.
  
  After doing a change of basis (and replacing $V$ by a subspace of
  the same dimension) we may assume that $B^t$ is in Jordan normal
  form with Jordan blocks $J(\lambda_1,n_1),\ldots,J(\lambda_m,n_m)$.
  As $B$ is similar to $B^t$ and as $B$ is Jordan regular by hypothesis,
  the eigenvalues $\lambda_1,\ldots,\lambda_m$
  are pairwise distinct. Thus
  \begin{equation*}
    \left(
      \begin{array}{lll}
        J(\lambda_1-z,n_1)^{-1}& & \\
                           & \ddots & \\
& &         J(\lambda_m-z,n_m)^{-1}\\
      \end{array}\right)\in V
  \end{equation*}
  for all $z\in K \setminus\{\lambda_1,\ldots,\lambda_m\}$.

  Let $\pi\colon \Mat_n(K)\rightarrow K^{n_1}\times \cdots\times
  K^{n_m}= K^n$ be the product of the projections  to
  the first row of
  each Jordan block $J(\cdot,n_1),\ldots,J(\cdot,n_m)$. The image of
  $V$ is a proper subspace of $ K^n$ as $\dim V<n$. Thus
  there are column vectors $\alpha_k\in K^{n_k}$, not all zero,
  with
  \begin{equation*}
    \sum_{k=1}^m \(J(\lambda_k-z,n_k)^{-1}\)_{1} \alpha_k =0
  \end{equation*}
  for all but finitely many $z\in K$;
  here the subscript $1$ means taking the first row.

  Let us be more concrete by using the formula
  \[
  (J(X,n_k)^{-1})_{1} = \(1/X,-1/X^2,1/X^3,\ldots,(-1)^{n_k-1} /X^{n_k}\). 
\]
  If we write $\alpha_k=(\alpha_{k,1},\ldots,\alpha_{k,n_k})$ then our
  relation becomes
  \begin{equation*}
    \sum_{k=1}^m \sum_{l=1}^{n_k}   \frac{\alpha_{k,l}}{(z-\lambda_k)^l}=0.
  \end{equation*}
  This equality holds for all but finitely many $z\in K$. We get a
  contradiction as $\lambda_1,\ldots,\lambda_m$ are pairwise distinct
  and as  $\alpha_1,\ldots,\alpha_m$ are not all zero.
\end{proof}

\subsection{Proof of Proposition~\ref{prop:degreenfac}}
\label{sec:Dim Vn}

%\begin{proof}[Proof of Proposition~\ref{prop:degreenfac}]

 Let us
abbreviate $\cV = \cV_{\ell_1,\ldots,\ell_n}$.

By basic facts from linear algebra there is a Jordan regular matrix
$B$ with characteristic polynomial $X^n +\ell_1 X^{n-1}+\cdots +
\ell_n$. 
The $K$-points of $\cV\cap \A_K^{n^2}$ correspond to
matrices in $\Mat_n(K)$ whose characteristic polynomial is
$X^n + \ell_1 X^{n-1}+\cdots + \ell_n$. 
Thus $B\in (\cV\cap\A_K^{n^2})(K)$.

The Jacobian criterion applied to the dehomogenized system
\[
  a_1(A) - \ell_1 , a_2(A) - \ell_2 , \ldots, a_n(A)-\ell_n,
\]
together with Lemma~\ref{lem:lin indep}, 
shows that $B$ satisfies the Jacobian condition for smoothness. 
By~\cite[Definition~6.14(1)]{GoertzWedhorn}, $\cV\cap
\A_K^{n^2}$
is smooth over $K$ of relative dimension $n^2-n$ 
at $B$.

In particular,
$\cV\cap\A_K^{n^2}$ is smooth of relative dimension
$n^2-n$ over $K$ at all Jordan regular matrices in $\Mat_n(K)$.
By~\cite[Lemma~6.26]{GoertzWedhorn}, the
 local ring of  $\cV$ at Jordan regular matrices is
 regular and has
dimension $n^2-n$. 
In particular, $\cV$ is reduced at all Jordan regular
matrices.

Above we have seen that $\cV$ admits at least one
regular (closed) point. The regular locus is Zariski open by a general
fact in this setting, see~\cite[Corollary~12.52(2)]{GoertzWedhorn}. So the local ring
of $\cV$ at the generic point is regular. The local
ring at the generic point has dimension $0$, 
see~\cite[Equation~(5.8.1)]{GoertzWedhorn}, and a
regular local ring of dimension $0$ is a field.  
The length of a field
as a module over itself is $1$. This $1$ is the multiplicity of
$\cV$ appearing in B\'ezout's Theorem, see~\cite[Corollary~2.5]{EisHar:3264} (and the follow-up comments)
and~\cite[Section~1.2.1]{EisHar:3264}.
As  $\cV$ is cut out by
hypersurfaces of degree $1,2,\ldots,n$ we obtain
$\deg \cV=n!$.

\section*{Acknowledgement}

The authors would like to thank Alexander Fish for the patient answering their questions 
about the abelianisation map, Boris Kunyavski  for very useful discussions
of problems on commutator subgroups, David Masser for communicating and allowing us to use his example~\eqref{eq:Bad Exp}  in Section~\ref{sec:com}, Peter Sarnak for informing us of~\cite{CRRZ} and Yuri Tschinkel for showing us the relevance of~\cite{C-LT}.

The authors  are  especially grateful to  Per Salberger  for patient explanation of  several issues related to 
the determinant method and to Alexei Skorobogatov  and
%%are especially grateful to  
Felipe Voloch for 
detailed clarifications of some algebraic geometry issues.

The first-named author acknowledges
support by  the Swiss National Science Foundation grant “Rational
points, arithmetic dynamics, and heights” (Nr. 200020\_219397). The second and third-named authors were  supported, in part, by the Australian Research Council
Grants DP200100355 and DP230100530.  The
second-named author gratefully acknowledges the hospitality and support of
the Max Planck Institute for Mathematics in Bonn, where parts of this
work has been carried out.

\end{document}